\documentclass{amsart}

\usepackage{amsmath,amssymb,amsfonts, amsthm}
\usepackage{verbatim}
\usepackage{mathrsfs}
\usepackage{url}
\usepackage[inline]{enumitem}
\usepackage{todonotes}
\usepackage[capitalize]{cleveref}

\usepackage{microtype}
\usepackage{tikz-qtree}
\usepackage{xspace}
\usetikzlibrary{automata,backgrounds,calc,positioning}

\tikzset{every tree node/.style={draw, circle, black, semithick, inner sep = 2pt, minimum size = 10pt}}
\tikzstyle{arrow} = [semithick,->,>=stealth]
\tikzset{edge from parent/.append style={arrow, edge from parent path = {(\tikzparentnode) -- (\tikzchildnode)}}}
\tikzset{level distance = 40pt, sibling distance = 20pt}
\tikzset{gate/.style={draw, circle, black, semithick, inner sep = 1.5pt}}

\numberwithin{equation}{section}

\theoremstyle{plain}
\newtheorem{thm}{Theorem}[section]
\newtheorem{prop}[thm]{Proposition}
\newtheorem{lem}[thm]{Lemma}
\newtheorem{cor}[thm]{Corollary}
\newtheorem{defin}[thm]{Definition}
\newtheorem{rem}[thm]{Remark}

\crefname{thm}{Theorem}{Theorems}
\Crefname{thm}{Theorem}{Theorems}
\crefname{prop}{Proposition}{Propositions}
\Crefname{prop}{Proposition}{Propositions}
\crefname{lem}{Lemma}{Lemmas}
\Crefname{lem}{Lemma}{Lemmas}
\crefname{cor}{Corollary}{Corollaries}
\Crefname{cor}{Corollary}{Corollaries}
\crefname{rem}{Remark}{Remarks}
\Crefname{rem}{Remark}{Remarks}

\newcommand{\Powerset}[1]{2^{#1}}

\newcommand{\cA}{\mathcal{A}}

\newcommand{\dG}{\mathbb{G}}
\newcommand{\dM}{\mathbb{M}}

\newcommand{\DLOGTIME}{\ensuremath{\mathsf{DLOGTIME}}}
\newcommand{\LogCFL}{\ensuremath{\mathsf{LogCFL}}}
\newcommand{\Poly}{\ensuremath{\mathsf{P}}}
\newcommand{\LogDCFL}{\ensuremath{\mathsf{LogDCFL}}}
\newcommand{\NP}{\mathsf{NP}}
\newcommand{\TCZero}{\mathsf{TC}^0}
\newcommand{\Periods}{\mathbb{P}}
\newcommand{\CompPeriods}{\mathbb{P}}

\newcommand{\N}{\mathbb{N}}
\newcommand{\Z}{\mathbb{Z}}
\newcommand{\x}{\mathsf{x}}
\newcommand{\y}{\mathsf{y}}
\newcommand{\mixedNorm}[1]{\|#1\|_{\mathsf{m}}}

\begin{document}

\title{The Complexity of Knapsack in Graph Groups}

\author{Markus Lohrey}
\address{Universit\"at Siegen, Germany}
\email{lohrey@eti.uni-siegen.de}

\author{Georg Zetzsche}
\address{LSV, CNRS \& ENS Cachan, Universit\'{e} Paris-Saclay, France}
\thanks{The second author is supported by a fellowship within the
Postdoc-Program of the German Academic Exchange Service (DAAD)}
\email{zetzsche@lsv.fr}

\begin{abstract}
Myasnikov et al. have introduced the knapsack problem for arbitrary finitely
generated groups.  In \cite{LohreyZ16} the authors proved that for each graph
group, the knapsack problem can be solved in {\sf NP}. Here, we determine the
exact complexity of the problem for every graph group. While the problem is
$\TCZero$-complete for complete graphs, it is $\LogCFL$-complete for each
(non-complete) transitive forest. For every remaining graph, the problem is
$\NP$-complete.
\end{abstract}

\maketitle

\section{Introduction}

In their paper \cite{MyNiUs14}, Myasnikov, Nikolaev, and Ushakov started the investigation of classical 
discrete optimization problems, which are classically formulated over the integers,
for arbitrary (possibly non-commutative) groups. The general goal of this line of research is to study 
to what extent results from the classical commutative setting can be transferred to the non-commutative setting.
Among other problems, Myasnikov et al.~introduced for a finitely generated group $G$ 
the {\em knapsack problem} and the {\em subset sum problem}.
The input for the knapsack problem is a sequence of group elements $g_1, \ldots, g_k, g \in G$ (specified
by finite words over the generators of $G$) and it is asked whether there exists a solution
$(x_1, \ldots, x_k) \in \mathbb{N}^k$
of the equation $g_1^{x_1} \cdots g_k^{x_k} = g$. For the subset sum problem one restricts the solution
to $\{0,1\}^k$.

For the particular case $G = \mathbb{Z}$  (where the additive notation 
$x_1 \cdot g_1 + \cdots + x_k \cdot g_k = g$ is usually preferred)
these problems are {\sf NP}-complete if the numbers $g_1, \ldots, g_k,g$ are 
encoded in binary representation. For subset sum, this is a classical result from Karp's seminal paper \cite{Karp72} on 
{\sf NP}-completeness. Knapsack for integers is usually formulated in a more general form in the literature;
{\sf NP}-completeness of the above form (for binary encoded integers) was shown  in \cite{Haa11}, where 
the problem was called \textsc{multisubset sum}.\footnote{Note that if we ask for a solution $(x_1,\ldots, x_k)$
in $\mathbb{Z}^k$, then knapsack can be solved in polynomial time (even for binary encoded integers) by checking whether
$\mathrm{gcd}(g_1, \ldots, g_k)$ divides $g$.}
Interestingly, if we consider subset sum  for the group 
$G = \mathbb{Z}$, but encode the input numbers  $g_1, \ldots, g_k, g$ 
in unary notation, then the problem is in {\sf DLOGTIME}-uniform
$\mathsf{TC}^0$ (a small subclass of polynomial time and even of logarithmic space that captures the complexity
of multiplication of binary encoded numbers; see e.g.
the book \cite{Vol99} for more details) \cite{ElberfeldJT11}, and 
the same holds for knapsack (see \cref{thm-TC0}). Related results can be found in 
\cite{Jen95}.

In \cite{MyNiUs14} the authors encode elements of the finitely generated group $G$ by words over the group generators
and their inverses, which corresponds to the unary encoding of integers.
Another, more succinct encoding of group elements uses \emph{straight-line programs} (SLP). These are context-free
grammars that produce a single word. Over a unary alphabet, one can achieve for every word exponential compression with SLPs:
The word $a^n$ can be produced by an SLP of size $O(\log n)$. This shows that knapsack and subset sum for the group $\mathbb{Z}$ with SLP-compressed
group elements correspond to the classical knapsack and subset sum problem with binary encoded numbers.
To distinguish between the two variants, we will speak in this introduction of uncompressed knapsack (resp., subset sum) if the input group
elements are given explicitly by words over the generators. On the other hand, if these words are represented by SLPs, we
will speak of SLP-compressed knapsack (resp., subset sum). Later in this paper, we will only use the uncompressed versions,
and denote these simply with knapsack and subset sum.

In our recent paper \cite{LohreyZ16}, we started to investigate knapsack and subset sum for graph groups, which are also known
as right-angled Artin groups in group theory. A graph group is specified by a finite simple graph $\Gamma$ and denoted with
$\dG(\Gamma)$. The vertices are the generators of the group, and two generators
$a$ and $b$ are allowed to commute if and only if $a$ and $b$ are adjacent in $\Gamma$. Graph groups interpolate between 
free groups and free abelian groups and can be seen as a group counterpart of trace monoids (free partially commutative
monoids), which have been used for the specification of concurrent behavior. In combinatorial group theory, graph groups
are currently an active area of research, mainly because of their rich subgroup structure, see e.g. \cite{BesBr97,CrWi04,GhPe07}.    

\subsection*{Contributions.}
In \cite{LohreyZ16} the authors proved that for every graph group, SLP-com\-pressed knapsack (resp., subset sum) is
{\sf NP}-complete. This result generalizes the classical result for knapsack with binary encoded integers.
Moreover, we proved that uncompressed knapsack and subset sum are {\sf NP}-complete for the group 
$F_2 \times F_2$ ($F_2$ is the free group on two
generators). The group $F_2 \times F_2$ is the graph group $\dG(\Gamma)$, where the graph $\Gamma$ is a cycle on four nodes.
This result leaves open the complexity of uncompressed knapsack (and subset sum) 
for graph groups, whose underlying graph does not contain an induced cycle on four nodes. In this paper, we (almost) completely
settle this open problem by showing the following results:
\begin{enumerate}[label=(\roman*)]
\item \label{i-TC0} Uncompressed knapsack and subset sum for $\dG(\Gamma)$ are complete for $\mathsf{TC}^0$ if $\Gamma$ is a complete graph (and thus
$\dG(\Gamma)$ is a free abelian group).\footnote{In the following, $\mathsf{TC}^0$ always refers to its {\sf DLOGTIME}-uniform version.}
\item \label{ii-logcfl} Uncompressed knapsack and subset sum for $\dG(\Gamma)$ are $\mathsf{LogCFL}$-complete if $\Gamma$ is not a complete graph and 
neither contains an induced cycle on four nodes ($\mathsf{C4}$) nor an induced path on four nodes ($\mathsf{P4}$).
\item \label{iii-NP} Uncompressed knapsack for $\dG(\Gamma)$ is $\mathsf{NP}$-complete if $\Gamma$ 
contains an induced $\mathsf{C4}$ or an induced $\mathsf{P4}$ (it is not clear whether this also holds for subset sum).
\end{enumerate}
\subsection*{Overview of the proofs.}
The result \labelcref{i-TC0} is a straightforward extension of the corresponding
result for $\mathbb{Z}$ \cite{ElberfeldJT11}.  The statements in
\labelcref{ii-logcfl} and  \labelcref{iii-NP} are less obvious. Recall that
$\mathsf{LogCFL}$ is the closure of the context-free languages under logspace
reductions; it is contained in $\mathsf{NC}^2$. 

To show the upper bound in \labelcref{ii-logcfl}, we use the fact that the graph
groups $\dG(\Gamma)$, where $\Gamma$ neither contains an induced $\mathsf{C4}$ nor
an induced $\mathsf{P4}$ (these graphs are the so called transitive forests),
are exactly those groups that can be built up from $\mathbb{Z}$ using the
operations of free product and direct product with $\mathbb{Z}$. We then
construct inductively over these operations a logspace-bounded auxiliary
pushdown automaton working in polynomial time (these machines accept exactly
the languages in $\mathsf{LogCFL}$) that checks whether an acyclic finite
automaton accepts a word that is trivial in the graph group. In order to apply
this result to knapsack, we finally show that every solvable knapsack instance
over a graph group $\dG(\Gamma)$ with $\Gamma$ a transitive forest has a
solution with polynomially bounded exponents. This is the most difficult result in this 
paper and it might be of
independent interest.

For the lower bound in  \labelcref{ii-logcfl}, it suffices to consider the
group $F_2$ (the free group on two generators). Our proof is based on the fact
that the context-free languages are exactly those languages that can be
accepted by valence automata over $F_2$. This is a reinterpretation of the
classical theorem of Chomsky and Sch\"utzenberger. To the authors' knowledge,
the result~\labelcref{ii-logcfl} is the first completeness result for
$\mathsf{LogCFL}$ in the area of combinatorial group theory.

Finally, for the result \labelcref{iii-NP} it suffices to show
$\mathsf{NP}$-hardness of knapsack for the graph group $\dG(\mathsf{P4})$ (the
$\mathsf{NP}$ upper bound and the lower bound for $\mathsf{C4}$ is shown in
\cite{LohreyZ16}). We apply a technique that was first used in a paper by
Aalbersberg and Hoogeboom \cite{AaHo89} to show that the intersection
non-emptiness problem for regular trace languages is undecidable for
$\mathsf{P4}$.

\subsection*{Related work.}

Apart from the above mentioned papers \cite{LohreyZ16,MyNiUs14}, also the papers 
\cite{FrenkelNU15,Godin16,KoeLoZe16} deal with the knapsack and subset sum problem for groups.
In \cite{FrenkelNU15} the authors study the effect of free and direct product constructions on the complexity
of the subset sum and knapsack problem. In \cite{KoeLoZe16} the authors proved that knapsack is undecidable
for the direct product of sufficiently many copies of the discrete Heisenberg group (this direct product is nilpotent
of class 2). On the other hand, knapsack is shown to be decidable for the discrete Heisenberg group, which implies
that decidability of knapsack is not preserved by direct products. Moreover, it is shown in \cite{KoeLoZe16} that
knapsack is decidable for every co-context-free group and that there exist polycyclic groups with an {\sf NP}-complete
subset sum problem.
Using the results of  \cite{KoeLoZe16}, Godin proved in \cite{Godin16} that there are automaton groups
with an undecidable knapsack problem. On the other hand for so called bounded automaton groups,
Godin proved decidability of the knapsack problem. The class of bounded automaton groups includes
important groups like for instance Grigorchuk's group, Gupta-Sidki groups, and the Basilica group.

\section{Knapsack and Exponent Equations}
 
 We assume that the reader has some basic knowledge concerning (finitely generated) groups; see e.g. \cite{LySch77}
 for further details. Let $G$ be a finitely generated group, and let $A$ be a finite generating set for $G$. Then,
 elements of $G$ can be represented by finite words over the alphabet $A^{\pm 1} = A \cup A^{-1}$.
An {\em exponent equation} over $G$ is an equation of the
form $$h_0 g_1^{\x_1} h_1 g_2^{\x_2}  h_2 \cdots g_k^{\x_k} h_k = 1$$ 
where
$g_1, g_2, \ldots, g_k, h_0, h_1, \ldots, h_k \in G$ are group elements that
are given by finite words over the alphabet $A^{\pm 1}$
and $\x_1, \x_2, \ldots, \x_k$
are not necessarily distinct variables. Such an exponent equation is \emph{solvable}
if there exists a mapping $\sigma \colon \{\x_1, \ldots,\x_k\} \to \mathbb{N}$ such that
$h_0 g_1^{\sigma(\x_1)} h_1 g_1^{\sigma(\x_2)}  h_2 \cdots g_k^{\sigma(\x_k)} h_k = 1$
in the group $G$. 
The \emph{size} of an
equation is $\sum_{i=0}^k |h_i|+\sum_{i=1}^k |g_i|$, where $|g|$
denotes the length of a shortest word $w\in (A^{\pm 1})^*$ representing $g$.
{\em Solvability of  exponent equations over $G$} is the following computational problem:

\smallskip
\noindent
{\bf Input:} An exponent equation $E$ over $G$ (where elements of $G$ are specified by words over the 
group generators and their inverses).

\smallskip
\noindent
{\bf Question:} Is $E$ solvable?

\smallskip
\noindent
The {\em knapsack problem} for the group $G$ is the restriction of solvability
of  exponent equations over $G$ to exponent equations of the form $g_1^{\x_1}
g_2^{\x_2} \cdots g_k^{\x_k} g^{-1} = 1$ or, equivalently, $g_1^{\x_1}
g_2^{\x_2} \cdots g_k^{\x_k} = g$ where the exponent variables $\x_1, \ldots,
\x_k$ have to be pairwise different. The {\em subset sum problem} for the group
$G$ is defined in the same way as the knapsack problem, but the exponent
variables $\x_1, \ldots, \x_k$ have to take values in $\{0,1\}$. 

It is a simple observation  that the decidability and complexity of solvability of  exponent equations over $G$ as well as  
the knapsack problem and subset sum problem for $G$ does not depend on the chosen finite generating set for the group $G$.  
Therefore, we do not have to mention the generating set explicitly in these problems. 

\begin{rem} \label{remark}
Since we are dealing with a group, one might also allow solution mappings $\sigma \colon  \{\x_1, \ldots,\x_k\} \to \mathbb{Z}$
to the integers. This variant of solvability of exponent equations (knapsack, respectively) can be reduced to the above
version, where $\sigma$ maps to $\mathbb{N}$, by simply replacing a power $g_i^{\x_i}$ by $g_i^{\x_i} (g^{-1}_i)^{\y_i}$, where
$\y_i$ is a fresh variable.
\end{rem}
 
\section{Traces and Graph Groups}

Let $(A,I)$ be a finite simple graph. In other words, the edge relation $I \subseteq A \times A$ is 
irreflexive and symmetric. It is also called the {\em independence relation}, and $(A,I)$ is called
an {\em independence alphabet}. Symbols $a,b \in A$ are called dependent if $(a,b) \not\in I$.
We consider the monoid $\dM(A, I) = A^*/\!\!\equiv_I$,
where $\equiv_I$ is the smallest congruence relation on the free monoid $A^*$ that contains all pairs
$(ab, ba)$ with $a,b \in A$ and $(a,b) \in I$. This monoid is called a {\em trace monoid} or {\em partially
commutative free monoid}. Elements of $\dM(A, I)$ are called 
{\em Mazurkiewicz traces} or simply {\em traces}. The trace represented by the word $u$ is denoted
by $[u]_I$, or simply $[u]$ if no confusion can arise. 
The empty trace $[\varepsilon]_I$ is the identity element of the monoid $\dM(A,I)$ 
and is denoted by $1$.
For a language $L \subseteq A^*$
we denote with $[L]_I = \{ [u]_I \in \dM(A,I) \mid u \in L \}$ the set of traces represented by $L$.  

Figure~\ref{fig:P4-C4} shows two important
indepedence alphabets that we denote with $\mathsf{P4}$ (path on four nodes) and $\mathsf{C4}$ (cycle on four nodes).
Note that $\dM(\mathsf{C4}) = \{a,c\}^* \times \{b,d\}^*$.

\begin{figure}[t]	
	\centering
	\tikzstyle{every node}=[minimum size=15pt]
	\begin{tikzpicture}
			
	\node[gate] (a) {$a$};
	\node[gate, right = 20pt of a] (b) {$b$};
	\node[gate, right = 20pt of b] (c) {$c$};
	\node[gate, right = 20pt of c] (d) {$d$};
	
	\draw[arrow,-] (a) -- (b);
	\draw[arrow,-] (b) -- (c);
	\draw[arrow,-] (c) -- (d);
	
	\node[gate, below right = 10 pt and 40pt of d] (a') {$a$};
	\node[gate, right = 20pt of a'] (b') {$b$};
	\node[gate, above = 20pt of b'] (c') {$c$};
	\node[gate, above = 20pt of a'] (d') {$d$};
	
	\draw[arrow,-] (a') -- (b');
	\draw[arrow,-] (b') -- (c');
	\draw[arrow,-] (c') -- (d');
        \draw[arrow,-] (d') -- (a');
	
	\end{tikzpicture}
	\caption{$\mathsf{P4}$ and $\mathsf{C4}$}
	\label{fig:P4-C4}
\end{figure}

With an independence alphabet $(A,I)$ we associate the finitely presented group 
\begin{equation*}
\dG(A,I) = \langle A \mid ab=ba \ ( (a,b) \in I) \rangle. 
\end{equation*}
More explicitly, this group can be defined as follows: 
Let $A^{-1} = \{ a^{-1} \mid a \in A\}$ be a disjoint copy of the alphabet $A$.
We extend the independence relation $I$ to $A^{\pm 1} = A \cup A^{-1}$ by
$(a^x, b^y) \in I$ for all $(a,b) \in I$ and $x,y \in \{-1,1\}$. 
Then 
$\dG(A,I)$ is the quotient monoid $(A^{\pm 1})^*/\sim_I$, where $\sim_I$ is the 
smallest congruence relation that contains (i) all pairs $(ab,ba)$ for 
$a,b \in A^{\pm 1}$ with $(a,b) \in I$ and (ii) all pairs $(aa^{-1}, \varepsilon)$ and
$(a^{-1}a, \varepsilon)$ for $a \in A$. Because of (ii), $\dG(A,I)$ is a group.

A group $\dG(A,I)$ is called a {\em graph group}, or
{\em right-angled Artin group}\footnote{This term comes from the fact that right-angled Artin groups
are exactly  the Artin groups corresponding to right-angled Coxeter groups.}, or {\em free partially commutative group}.
Here, we use the term graph group.
Graph groups received a lot of attention in group theory during the last years, mainly
due to their rich subgroup structure \cite{BesBr97,CrWi04,GhPe07}, and their relationship to low dimensional
topology (via so-called virtually special groups) \cite{Agol12,HagWi10,Wis09}.

In this paper we will classify graph groups with respect to the complexity of knapsack.
It turns out that for every graph group, knapsack is complete for one of the following three
complexity classes: $\TCZero$, $\LogCFL$, and $\NP$.
We assume that the reader has some basic knowledge in complexity theory, in particular
concerning the class $\NP$. In this next section, we briefly recall the definitions of the other 
two classes, $\TCZero$ and $\LogCFL$; further details can be found in \cite{Vol99}.

\section{$\TCZero$ and $\LogCFL$}. 

The class $\TCZero$ is a very low circuit complexity class; it is contained for instance in $\mathsf{NC}^1$ and deterministic
logspace. We will use this class only in Section~\ref{sec-tc0-completeness}, and even that part can be understood without
the precise definition of $\TCZero$. Nevertheless, for completeness we include the formal definition of $\TCZero$.

A language $L \subseteq \{0,1\}^*$ 
belongs to $\TCZero$ if there exists a family $(C_n)_{n \geq 0}$ of Boolean circuits with the following 
properties:
\begin{itemize}
\item $C_n$ has $n$ distinguished input gates $x_1, \ldots, x_n$ and a distinguished output gate $o$.
\item $C_n$ accepts exactly the words from $L \cap \{0,1\}^n$, i.e., if the input gate $x_i$ receives the input $a_i \in \{0,1\}$, then
the output gate $o$ evaluates to $1$ if and only if $a_1 a_2 \cdots a_n \in L$.
\item Every circuit $C_n$ is built up from input gates, and-gates, or-gates, and majority-gates, where a majority gate evaluates to $1$
if more than half of its input wires carry $1$.
\item  All gates have unbounded fan-in, which means that there is no bound on the number of input wires for a gate.
\item There is a polynomial $p(n)$ such that $C_n$ has at most $p(n)$ many gates.
\item There is a constant $c$ such that every $C_n$ has depth at most $c$, where the depth is the length of a longest path
from an input gate $x_i$ to the output gate $o$.
\end{itemize}
This is in fact the definition of non-uniform $\TCZero$. Here ``non-uniform'' means that the mapping $n \mapsto C_n$ is 
not restricted in any way. In particular, it can be non-computable. For algorithmic purposes one usually adds some uniformity
requirement to the above definition. The most ``uniform'' version of $\TCZero$ is $\DLOGTIME$-uniform $\TCZero$. For this,
one encodes the gates of each circuit $C_n$ by bit strings of length $O(\log n)$. Then the circuit family $(C_n)_{n \geq 0}$
is called \emph{$\DLOGTIME$-uniform}  if (i) there exists a deterministic Turing machine that computes for a given gate $u \in \{0,1\}^*$ 
of $C_n$ ($|u| \in O(\log n)$) in time $O(\log n)$ the type (of gate $u$, where the types are $x_1, \ldots, x_n$, and, or, majority)
and (ii) there exists a deterministic Turing machine that decides for two given gate $u,v \in \{0,1\}^*$ 
of $C_n$ ($|u|, |v| \in O(\log n)$) in time $O(\log n)$ whether there is a wire from gate $u$ to gate $v$.
In the following, we always implicitly refer to $\DLOGTIME$-uniform $\TCZero$.

If the language $L$ in the above definition of $\TCZero$ is defined over a non-binary alphabet $\Sigma$ then one first has to fix a binary 
encoding of words over $\Sigma$. When talking about hardness for $\TCZero$, one has to use reductions, whose computational power
are below $\TCZero$, e.g. $\mathsf{AC}^0$-reductions. The precise definition of these reductions is not important for our purpose.
Important problems that are complete for $\TCZero$ are:
\begin{itemize}
\item The languages $\{ w \in \{0,1\}^* \mid |w|_0 \leq |w|_1 \}$ and $\{ w \in \{0,1\}^* \mid |w|_0 = |w|_1 \}$, where $|w|_a$ denotes
the number of occurrences of $a$ in $w$, see e.g. \cite{Vol99}.
\item The computation (of a certain bit) of the binary representation of the product of two (or any number of) 
binary encoded integers \cite{HeAlBa02}.
\item The computation (of a certain bit) of the binary representation of the integer quotient of two binary encoded integers \cite{HeAlBa02}.
\item The word problem for every infinite solvable linear group \cite{KoeLo15cocoon}.
\item The conjugacy problem for the Baumslag-Solitar group  $\mathsf{BS}(1,2)$  \cite{DiekertMW14}.
\end{itemize}
The class $\LogCFL$ consists of all
problems that are logspace reducible to a context-free language. 
The class $\LogCFL$ is included in the parallel complexity class
$\mathsf{NC}^2$ and has several alternative characterizations (see e.g.
\cite{Su78,Vol99}):
\begin{itemize}
\item logspace bounded alternating Turing-machines with polynomial proof tree size,
\item semi-unbounded Boolean circuits of polynomial size and logarithmic depth, and
\item logspace bounded auxiliary pushdown automata with polynomial running time.  
\end{itemize}
For our purposes, the last characterization is most suitable.
An AuxPDA (for auxiliary pushdown automaton)
is a nondeterministic pushdown automaton with a two-way input tape and 
an additional work tape. Here we only consider AuxPDA with the following two restrictions:
\begin{itemize}
\item The length of the work tape is restricted to $O(\log n)$ for an input of length $n$
(logspace bounded).
\item There is a polynomial $p(n)$, such that every computation path
of the AuxPDA on an input of length $n$ has length at most $p(n)$
(polynomially time bounded).
\end{itemize}
Whenever we speak of an AuxPDA in the following, we implicitly assume that the AuxPDA
is logspace bounded and polynomially time bounded.
Deterministic AuxPDA are defined in the obvious way.
The class of languages that are accepted by AuxPDA is exactly $\LogCFL$, whereas the class of langauges accepted by
deterministic AuxPDA is $\LogDCFL$ (the closure of the deterministic context-free languages under logspace reductions)
\cite{Su78}.

\section{$\TCZero$-completeness} \label{sec-tc0-completeness}

Let us first consider subset sum and knapsack for free abelian groups $\mathbb{Z}^m$.  Note that
$\mathbb{Z}^m$ is isomorphic to the graph group $\dG(A,I)$ where $(A,I)$ is the
complete graph on $m$ nodes. Our first result is a simple combination of known
results~\cite{ElberfeldJT11,Papadimitriou81}.

\begin{thm}\label{thm-TC0}
For every fixed $m\ge 1$,  knapsack and  subset sum for
the free abelian group $\mathbb{Z}^m$ are complete for $\TCZero$. Hence,
 knapsack and  subset sum
for $\dG(A,I)$ are complete for $\TCZero$ if $(A,I)$ is a non-empty
complete graph.
\end{thm}

\begin{proof}
Hardness for $\TCZero$ follows from the fact that the word problem for
$\mathbb{Z}$ is $\TCZero$-complete: this is exactly the problem of checking
whether for a given word $w \in \{0,1\}^*$, $|w|_0 = |w|_1$ holds.

Let us now show that  knapsack for $\mathbb{Z}^m$ belongs to
$\TCZero$.  Let $A = \{a_1,\ldots, a_m\}$ be the generating set for
$\mathbb{Z}^m$. Given a word $w \in (A \cup A^{-1})^*$ we can compute the
vector $(b_1, \ldots, b_m) \in \mathbb{Z}^m$ with $b_i := |w|_{a_i} -
|w|_{a_i^{-1}}$ represented in unary notation in $\TCZero$ (counting the number
of occurrences of a symbol in a string and subtraction can be done in
$\TCZero$).  Hence, we can transform in $\TCZero$ an instance of 
knapsack for $\mathbb{Z}^m$ into a system of equations $A x = b$, where $A \in
\mathbb{Z}^{m \times n}$ is an integer matrix with unary encoded entries, $b
\in \mathbb{Z}^m$ is an integer vector with unary encoded entries, and $x$ is a
vector of $n$ variables ranging over $\mathbb{N}$. Let $t = n (ma)^{2m+1}$,
where $a$ is the maximal absolute value of an entry in $(A\mid b)$.  By
\cite{Papadimitriou81} the system $Ax=b$ has a solution if and only if it has a
solution with all entries of $x$ from the interval $[0, t]$. Since $m$ is a
constant, the unary encoding of the number $t$ can be computed in $\TCZero$
(iterated multiplication can be done in $\TCZero$).  However, the question whether
the system $Ax=b$ has a solution from $[0, t]^n$ is an instance of the
$m$-integer-linear-programming problem from \cite{ElberfeldJT11}, which was
shown to be in $\TCZero$ in \cite{ElberfeldJT11}. For  subset sum
for $\mathbb{Z}^m$ one can use the same argument with $t=1$.
\end{proof}

 \section{$\LogCFL$-completeness}

In this section, we characterize those graph groups where knapsack for
$\dG(A,I)$ is $\LogCFL$-complete. 

The \emph{comparability graph} of a tree $t$ is the simple graph with the same
vertices as $t$, but has an edge between two vertices whenever one is a
descendent of the other in $t$. A graph $(A,I)$ is a \emph{transitive forest}
if it is a disjoint union of comparability graphs of trees. 
Wolk characterized transitive forests as those graphs that neither contain 
$\mathsf{P4}$ nor $\mathsf{C4}$ (see Figure~\ref{fig:P4-C4}) as an induced subgraph \cite{Wolk65}. 
The main result of this section is:

\begin{thm}\label{logcfl-completeness}
If $(A,I)$ is a transitive forest and not complete, then knapsack and subset sum for $\dG(A,I)$ are $\LogCFL$-complete.
\end{thm}

If the graph $(A,I)$ is the disjoint union of graphs $\Gamma_0$ and $\Gamma_1$,
then by definition, we have $\dG(A,I)\cong\dG(\Gamma_0)*\dG(\Gamma_1)$. If one
vertex $v$ of $(A,I)$ is adjacent to every other vertex and removing $v$ from
$(A,I)$ results in the graph $\Gamma_0$, then $\dG(A,I)\cong
\dG(\Gamma_0)\times\Z$. Therefore, we have the following \emph{inductive
characterization} of the graph groups $\dG(A,I)$ for transitive forests $(A,I)$:
It is the smallest class of groups containing the trivial group that is closed
under taking \begin{enumerate*}[label=(\roman*)]\item free products and \item
direct products with $\Z$.\end{enumerate*} 

\subsection*{Acyclic Automata}
In the proofs for the upper bound as well as  the lower bound in Theorem~\ref{logcfl-completeness}
we employ the membership problem for acyclic automata, which has already been studied in connection with
the knapsack and subset sum problem~\cite{FrenkelNU15,KoeLoZe16}.

We define a {\em finite automaton} as a tuple $\mathcal{A} = (Q,\Sigma, \Delta,
q_0, q_f)$, where $Q$ is a finite set of states, $\Sigma$ is the \emph{input
alphabet}, $q_0 \in Q$ is the \emph{initial state}, $q_f \in Q$ is the \emph{final state}, and
$\Delta \subseteq Q \times \Sigma^* \times Q$ is a finite set of \emph{transitions}.
\iffalse
Note that we allow arbitrary finite words over the input alphabet as
transition labels. Since also the empty word is allowed we can restrict to
finite automata with a unique final state. 
\fi
The language accepted by
$\mathcal{A}$ is denoted with $L(\mathcal{A})$.  An {\em acyclic automaton} is
a finite automaton $\mathcal{A} = (Q,\Sigma, \Delta, q_0, q_f)$ such that the
relation $\{ (p,q) \mid \exists w \in \Sigma^* : (p,w,q) \in \Delta \}$ is
acyclic.  
For a graph group $\dG(A,I)$ the  \emph{membership problem for acyclic automata}
is the following computational problem:

\smallskip
\noindent
{\bf Input:}  An acyclic automaton $\mathcal{A}$  over the input alphabet $A \cup A^{-1}$.

\smallskip
\noindent
{\bf Question:} Is there a word $w \in L(\mathcal{A})$ such that $w=1$ in $\dG(A,I)$?

\smallskip
\noindent

In order to show the upper bound in \cref{logcfl-completeness}, we 
reduce knapsack for $\dG(A,I)$ with  $(A,I)$ a transitive forest to 
the membership problem for acyclic automata for $\dG(A,I)$ (note that for subset sum
this reduction is obvious). Then, we
apply the following \lcnamecref{thm-acyclic-logcfl}.  From work of Frenkel,
Nikolaev, and Ushakov~\cite{FrenkelNU15}, it follows that the membership problem for acyclic automata
is in $\Poly$. We strengthen this to $\LogCFL$:
\begin{prop} \label{thm-acyclic-logcfl}
If $(A,I)$ is a transitive forest, then the membership problem for acyclic automata over $\dG(A,I)$ is in $\LogCFL$.
\end{prop}
The proof of \cref{thm-acyclic-logcfl} uses the following lemma:

\begin{lem} \label{lemma-auxpda}
For every transitive forest $(A,I)$ with the associated graph group $G = \mathbb{G}(A,I)$
there is a deterministic AuxPDA $\mathcal{P}(G)$ with input alphabet $A^{\pm 1}$ and the following properties:
\begin{itemize}
\item In each step, the input head for $\mathcal{P}(G)$ either does not move, or moves one step to the right.
\item If the input word is equal to $1$ in $G$, then $\mathcal{P}(G)$ terminates  in the distinguished
state $q_1$ with empty stack. Let us call this state the $1$-state of $\mathcal{P}(G)$.
\item  If the input word is not equal to $1$ in $G$, then $\mathcal{P}(G)$ terminates  
in a state different from $q_1$ (and the stack is not necessarily empty). 
\end{itemize}
\end{lem}

\begin{proof}
We construct the AuxPDA $\mathcal{P}(G)$ by induction over the structure of the group $G$. For this, we consider
the three cases that $G = 1$, $G = G_1 * G_2$, and $G  = \mathbb{Z} \times G'$. The case that 
$G = 1$ is of course trivial.

\medskip
\noindent
{\em Case} $G =  \mathbb{Z} \times G'$. We have already constructed the AuxPDA $\mathcal{P}(G')$.
The AuxPDA $\mathcal{P}(G)$ simulates the AuxPDA $\mathcal{P}(G')$
on the generators of $G'$. Moreover, it stores the current value of the $\mathbb{Z}$-component in binary notation
on the work tape. If the input word has length $n$, then $O(\log n)$ bits are sufficient for this. At the end,
$\mathcal{P}(G)$ goes into its $1$-state if and only if  $\mathcal{P}(G')$ is in its $1$-state
(which implies that the stack will be empty) and the $\mathbb{Z}$-component is zero.

\medskip
\noindent
{\em Case}  $G = G_1 * G_2$. For $i \in \{1,2\}$,
we have already constructed the AuxPDA $\mathcal{P}_i = \mathcal{P}(G_i)$. Let
$A_i^{\pm 1}$ be its input alphabet, which is a monoid generating set for $G_i$.
Consider now an input word $w \in (A_1^{\pm 1} \cup A_2^{\pm 1})^*$.
Let us assume that $w = u_1 v_1 u_2 v_2 \cdots u_k v_k$ with $u_i \in (A_1^{\pm 1})^+$ and $v_i \in (A_2^{\pm 1})^+$ (other cases
can be treated analogously). The AuxPDA $\mathcal{P}(G)$ starts with empty stack and simulates the AuxPDA 
$\mathcal{P}_1$ on the prefix $u_1$. If it turns out that $u_1 = 1$ in $G_1$ (which means that $\mathcal{P}_1$ is 
in its $1$-state) then the stack will be empty and the AuxPDA $\mathcal{P}(G)$
continues with simulating $\mathcal{P}_2$ on $v_1$. On the other hand, if $u_1 \neq 1$ in $G_1$, then $\mathcal{P}(G)$
pushes the state together with the work tape content of $\mathcal{P}_1$ reached after reading $u_1$ on the stack (on top of 
the final stack content of $\mathcal{P}_1$). This allows  $\mathcal{P}(G)$ to resume the computation of $\mathcal{P}_1$ later. Then $\mathcal{P}(G)$
continues with simulating $\mathcal{P}_2$ on $v_1$. 

The computation of $\mathcal{P}(G)$ will continue in this way. More precisely, if after reading $u_i$ (resp. $v_i$ with $i < k$) the 
AuxPDA $\mathcal{P}_1$ (resp. $\mathcal{P}_2$) is in its $1$-state then either
\begin{enumerate}[label=(\roman*)]
\item the stack is empty or 
\item the top part of the stack is of the form $s q t$ ($t$ is the top), where $s$ is a stack content of $\mathcal{P}_2$
(resp. $\mathcal{P}_1$), $q$ is a state of $\mathcal{P}_2$
(resp. $\mathcal{P}_1$) and $t$ is a work tape content of $\mathcal{P}_2$
(resp. $\mathcal{P}_1$). 
\end{enumerate}
In case (i), $\mathcal{P}(G)$ continues with the simulation of $\mathcal{P}_2$
(resp. $\mathcal{P}_1$) on the word $v_i$ (resp. $u_{i+1}$) in the initial configuration.
In case (ii), $\mathcal{P}(G)$ continues with the simulation of $\mathcal{P}_2$
(resp. $\mathcal{P}_1$) on the word $v_i$ (resp. $u_{i+1}$), where the simulation is started with 
stack content $s$, state $q$, and work tape content $t$. On the other hand, if after reading $u_i$ (resp. $v_i$ with $i < k$) the 
AuxPDA $\mathcal{P}_1$ (resp. $\mathcal{P}_2$) is not in its $1$-state then $\mathcal{P}(G)$ pushes on the stack the state
and work tape content of $\mathcal{P}_1$ reached after its simulation on $u_i$. 
This concludes the description of the AuxPDA $\mathcal{P}(G)$. It is clear that $\mathcal{P}(G)$ has the properties
stated in the lemma.
\end{proof}
\noindent
We can now prove \cref{thm-acyclic-logcfl}:

\begin{proof}[Proof of \cref{thm-acyclic-logcfl}.]
 Fix the graph group $G = \mathbb{G}(A,I)$, 
where $(A,I)$ is a transitive forest. An AuxPDA for the membership problem for acyclic automata
guesses a path in the input automaton $\mathcal{A}$ and thereby simulates the AuxPDA $\mathcal{P}(G)$ from Lemma~\ref{lemma-auxpda}.
If the final state of the input automaton $\mathcal{A}$ is reached while the AuxPDA $\mathcal{P}(G)$ is in the accepting state $q_1$,
then the overall AuxPDA accepts. It is important that the AuxPDA $\mathcal{P}(G)$ works one-way since 
the guessed path in $\mathcal{A}$ cannot be stored in logspace. This implies that the AuxPDA cannot reaccess 
the input symbols that already have been processed.  Also note that the AuxPDA 
is logspace bounded and polynomially time bounded since $\mathcal{A}$
is acyclic.
\end{proof}

\subsection{Bounds on knapsack solutions}

As mentioned above, we reduce for graph groups $\dG(A,I)$ with $(A,I)$ a transitive forest
the knapsack problem to the membership problem for acyclic automata.
To this end, we show that every positive knapsack instance has a
polynomially bounded solution. The latter is the most involved proof in our
paper.  

Frenkel, Nikolaev, and Ushakov~\cite{FrenkelNU15} call groups with this
property \emph{polynomially bounded knapsack groups} and show that this class
is closed under taking free products. However, it is not clear if direct
products with $\Z$ also inherit this property and we leave this question open.

Hence, we are looking for a property that yields polynomial-size solutions and
is passed on to free products and to direct products with $\Z$. It is known
that the solution sets are always semilinear.  If $(A,I)$ is a transitive
forest, this follows from a more general semilinearity property of rational
sets~\cite{LohSte08} and for arbitrary graph groups, this was shown in
\cite{LohreyZ16}. Note that it is not true that the solution sets alwas have
polynomial-size semilinear representations. This already fails in the case of $\Z$:
The equation $\x_1+\cdots+\x_{k}=k$ has $\binom{2k-1}{k}\ge 2^k$ solutions.
We will show here that the solution sets have semilinear representations
where every occuring number is bounded by a polynomial.

For a vector $x=(x_1,\ldots,x_k)\in\Z^k$, we define the norm 
$$\|x\|=\max \{|x_i| \mid i\in[1,k]\}.$$
For a subset $T\subseteq\N^k$, we write $T^\oplus$ for
the smallest subset of $\N^k$ that contains $0$ and is closed under addition.
A subset $S\subseteq \N^k$ is called \emph{linear} if there is a vector
$x\in\N^k$ and a finite set $F\subseteq\N^k$ such that $S=x+F^\oplus$.  Note
that a set is linear if and only if it can be written as $x+A\N^t$ for some
$x\in\N^k$ and some matrix $A\in\N^{k\times t}$. Here, $A\N^t$ denotes the set
of all vectors $Ay$ for $y\in\N^t$.  A \emph{semilinear set} is a finite union
of linear sets. If $S=\bigcup_{i=1}^n x_i+F_i^\oplus$ for
$x_1,\ldots,x_n\in\N^k$ and finite sets $F_1,\ldots,F_n\subseteq\N^k$, then the
tuple $(x_1,F_1,\ldots,x_n,F_n)$ is a \emph{semilinear representation} of $S$
and the \emph{magnitude} of this representation is defined as the maximum of
$\|y\|$, where $y$ ranges over all elements of $\bigcup_{i=1}^n \{x_i\}\cup
F_i$.  The \emph{magnitude} of a semilinear set $S$ is the smallest magnitude
of a semilinear representation for $S$. 

\begin{defin}
A group $G$ is called \emph{knapsack tame} if there is a polynomial $p$ such
that for every exponent equation $h_0 g_1^{\x_1}h_1g_2^{\x_2}h_2\cdots
g_n^{\x_k}h_k=1$ of size $n$ with pairwise distinct variables
$\x_1,\ldots,\x_k$, the set $S\subseteq \N^k$ of solutions is semilinear
of magnitude at most $p(n)$.
\end{defin}

Observe that although the size of an exponent equation may depend on the chosen
generating set of $G$, changing the generating set increases the size only by a
constant factor. Thus, whether or not a group is knapsack tame is independent of
the chosen generating set.

\begin{thm}\label{tforest-tame}
If $(A,I)$ is a transitive forest, then $\dG(A,I)$ is knapsack tame.
\end{thm}

Note that \cref{tforest-tame} implies in particular that every solvable
exponent equation with pairwise distinct variables has a polynomially bounded solution. 
\cref{tforest-tame} and \cref{thm-acyclic-logcfl} easily yield the upper bound
in \cref{logcfl-completeness}, see \cref{tforest-logcfl-membership}.

We prove \cref{tforest-tame} by showing that knapsack tameness transfers from groups $G$
to $G\times\Z$ (\cref{tame-directz}) and from $G$ and $H$ to $G*H$
(\cref{tameness-free-preserved}). Since the trivial group is obviously knapsack
tame, the inductive characterization of groups $\dG(A,I)$ for transitive
forests $(A,I)$ immediately yields \cref{tforest-tame}.

\subsection{Tameness of direct products with $\Z$}\label{tame-directz-section}
In this \lcnamecref{tame-directz-section}, we show the following.
\begin{prop}\label{tame-directz}
If $G$ is knapsack tame, then so is $G\times\Z$.
\end{prop}

\subsection*{Linear Diophantine equations} 

Vectors will be column vectors in this section. For a vector $x \in \Z^k$ we denote
with $x^T$ the corresponding row vector.
We employ a result of
Pottier~\cite{Pottier1991}, which bounds the norm of minimal non-negative
solutions to a linear Diophantine equation.  Let $A\in\Z^{k\times m}$ be an
integer matrix where $a_{ij}$ is the entry of $A$ at row $i$ and column $j$ and
let $x\in\Z^m$ be a vector with entries $x_i$. We will use the following norms:
\begin{eqnarray*}
\|A\|_{1,\infty} &=& \max_{i\in[1,k]}(\sum_{j\in[1,m]} |a_{ij}|), \\
\|A\|_{\infty,1} &=& \max_{j\in[1,m]} (\sum_{i\in[1,k]} |a_{ij}|), \\
\|A\|_{\infty} &=& \max_{i\in [1,k], j\in[1,m]} |a_{ij}|, \\
\|x\|_1 &=& \sum_{i=1}^m |x_i| .
\end{eqnarray*}
Recall that $\|x\|=\max_{i\in[1,m]} |x_i|$. A non-trivial solution $x\in\N^m\setminus\{0\}$ to
the equation $Ax=0$ is \emph{minimal} if there is no $y\in\N^m\setminus\{0\}$ with $Ay=0$ and
$y\le x$, $y\ne x$. Here $y \le x$ means that $y_i \le x_i$ for all $i \in [1,m]$.
 The set of all solutions clearly forms a submonoid of
$\N^m$. Let $r$ be the rank of $A$. 

\begin{thm}[Pottier~\cite{Pottier1991}]\label{pottier}
Each non-trival minimal solution $x \in \N^m$ to $Ax=0$  satisfies
$\|x\|_1\le (1+\|A\|_{1,\infty})^r$.
\end{thm}

We only need \cref{pottier} for the case that $A$ is a row vector $u^T$ for $u \in \Z^k$.
\begin{cor} \label{c-pottier}
Let $u \in \Z^k$. Each 
non-trival minimal solution  $x \in \N^k$ to $u^T x=0$  satisfies
$\|x\|_1\le 1+\|u\|_1$.
\end{cor}

By applying \cref{pottier} to the row vector $(u^T, -b)$ for $b \in \Z$, it is easy to
deduce that for each $x\in\N^k$ with $u^T x=b$, there is a $y\in\N^k$ with $u^T y=b$,
$y\le x$, and $\|y\|_1\le 1+\| \binom{u}{b}\|_1 = 1+ \|u\|_1 + |b|$.
We reformulate \cref{c-pottier} as follows.
\begin{lem}\label{pottier-reform}
Let $u \in \Z^k$ and $b \in \Z$.
Then the set
$\{x\in\N^k \mid u^T x=b\}$ admits a decomposition
$\{ x\in\N^k \mid u^T x=b \} = \bigcup_{i=1}^{s} c_i + C\N^t$,
where $c_i\in\N^k$ and $C\in\N^{k\times t}$ with $\|c_i\|_1$ and $\|C\|_{\infty,1}$ bounded by $1+\|u\|_1+|b|$.
\end{lem}
\begin{proof}
Let $\{c_1,\ldots,c_s\}$ be the set of minimal solutions of $u^Tx=b$. Then, as explained above, \cref{c-pottier}
yields $\|c_i\|_1\le  1+ \|u\|_1 + |b|$.
Moreover, let $C\in\N^{k\times t}$ be the matrix whose columns are the non-trivial minimal solutions of
$u^T x=0$. Then we have $\|C\|_{\infty,1} \le 1+\|u\|_1$.
This clearly yields the desired decomposition.
\end{proof}
The problem is that we want to apply \cref{pottier-reform} in a situation where we
have no bound on $\|u\|_1$, but only one on $\|u\|$.  
The following lemma yields such a bound. 

\begin{lem}\label{sol-decomp-onedim}
If $u\in\Z^k$ and $b\in\Z$ with $\|u\|,|b|\le M$,
then we have a decomposition
$\{x\in\N^k \mid u^Tx=b \} = \bigcup_{i=1}^{s} c_i + C\N^t \label{sol-decomp-onedim-final}$
where $\|c_i\|_1, \|C\|_{\infty,1} \leq 1+(M+2)M$.
\end{lem}
\begin{proof}
Write $u^T=(b_1,\ldots,b_k)$ and consider the row vector $v^T = (b'_1,\ldots,b'_{2M+1})$, $v  \in\Z^{2M+1}$, with entries $b'_i = i-(M+1)$.
Thus, we have
\[ v^T =  (b'_1,\ldots,b'_{2M+1})=(-M,-M+1,\ldots,-1,0,1,\ldots,M) .\]
Moreover, define the matrix $S = (s_{ij}) \in\N^{(2M+1)\times k}$ with 
\[ s_{ij} = \begin{cases} 1 & \text{if $b_j=b'_i$,} \\ 0 & \text{otherwise.} \end{cases} \]
Then clearly $u=v^T S$ and $\|v\|_1 =(M+1)M$. Furthermore, observe that
we have $\|Sx\|_1=\|x\|_1$ for every $x\in\N^k$ and that for each $y\in\N^{2M+1}$ the set
\[ T_y = \{ x\in\N^k \mid Sx=y\} \]
is finite. According to \cref{pottier-reform}, we can write
\begin{equation} \{x\in\N^{2M+1} \mid v^T x=b \} = \bigcup_{i=1}^{s'} c'_i+C'\N^{t'} \label{sol-decomp-onedim-step}\end{equation}
where $\|c'_i\|_1, \|C'\|_{\infty,1} \leq 1+ (M+1)M + M  = 1+(M+2)M$.
Let $\{ c_1,\ldots,c_s\}$ be the union of all sets $T_{c'_i}$
for $i\in[1,s']$ and let $C\in\N^{k\times t}$ be the matrix whose columns
comprise all $T_v$ where $v\in\N^{2M+1}$ is a column of $C'$.  Since
we have $\|Sx\|_1=\|x\|_1$ for $x\in\N^k$, the vectors $c_i$ obey the same bound as the
vectors $c'_i$, meaning $\|c_i\|_1\le 1+(M+2)M$. By the same argument, we have
$\|C\|_{\infty,1}\le \|C'\|_{\infty,1}\le 1+(M+2)M$. It remains to be shown that 
the equality from the \lcnamecref{sol-decomp-onedim} holds.

Suppose $u^T x=b$. Then $v^TSx=b$ and hence $Sx=c'_i+C'y$ for some $y\in\N^{t'}$.
Observe that if $Sz=p+q$ for $p,q \in \N^{2M+1}$, then $z$ decomposes as $z=p'+q'$, $p',q' \in \N^k$, so that $Sp'=p$ and
$Sq'=q$. Therefore, we can write $x=x_0+\cdots +x_n$ with $Sx_0=c'_i$ and
$Sx_j$ is some column of $C'$ for each $j\in[1,n]$. Hence, $x_0=c_r$ for some $r\in[1,s]$ and for
each $j\in[1,n]$, $x_j$ is a column of $C$. This proves $x\in c_r + C\N^t$.

On the other hand, the definition of $c_1,\ldots,c_s$ and $C$ implies that
for each column $v$ of $C$, $Sv$ is a column of $C'$. Moreover, for each
$i\in[1,s]$, there is a $j\in[1,s']$ with $Sc_i=c'_j$ and thus
$Sc_i+SC\N^t\subseteq c'_j+C'\N^{t'}$. Therefore
\[ u^T(c_i+C\N^t) =v^TS(c_i+C\N^t)\subseteq v^T(c'_j+C'\N^{t'}) \]
and the latter set contains only $b$ because of \labelcref{sol-decomp-onedim-step}.
\end{proof}

\begin{lem}\label{semil-intersection}
Let $S\subseteq\N^k$ be a semilinear set of magnitude $M$ and $u\in\Z^k$, $b\in\Z$
with $\|u\|, |b| \le m$. Then $\{ x\in S \mid u^T x=b \}$
is a semilinear set of magnitude at most $(kmM)^d$ for some constant $d$.
\end{lem}
\begin{proof}
Let $T=\{x\in\N^k \mid u^Tx=b\}$.  We may assume that $S$ is linear of magnitude
$M$, because if $S=L_1\cup\cdots\cup L_n$, then $S\cap T=(L_1\cap T)\cup
\cdots\cup (L_n\cap T)$.

Write $S=a+A\N^n$ with $a\in\N^k$ and $A\in\N^{k\times n}$, where
$\|a\|\le M$ and $\|A\|_\infty\le M$.  Consider the set $U=\{x\in\N^n \mid
u^T Ax=b-u^T a\}$. Note that $u^T A\in\Z^{1\times n}$ and 
\[ \|u^TA\| \le k\cdot \|u\| \cdot \|A\|_\infty \le kmM, \]
\[ |b-u^Ta| \le m+k\cdot \|u\| \cdot \|a\|\le m+kmM. \]
According to \cref{sol-decomp-onedim}, we can write $U=\bigcup_{i=1}^s
c_i+C\N^t$ where $\|c_i\|_1$ and $\|C\|_{\infty,1}$ are at most
$1+ (m+kmM)(m+kmM+2)$. Observe that
\[ a+AU = \bigcup_{i=1}^s a+Ac_i+AC\N^t \]
and 
\[ \|a+Ac_i\|\le \|a\| + \|A\|_\infty\cdot \|c_i\|_1 \le  2M + M (m+kmM)(m+kmM+2) , \]
\[ \| AC \|_{\infty} \le \|A\|_\infty \cdot \|C\|_{\infty,1} \le  M + M (m+kmM)(m+kmM+2). \]
Finally, note that $S\cap T=a+AU$.
\end{proof}

We are now ready to prove \cref{tame-directz}.
\begin{proof}[Proof of \cref{tame-directz}]
Suppose $G$ is knapsack tame with polynomial $\bar{p}$.
Let 
\begin{equation}
h_0g_1^{\x_1}h_1g_2^{\x_2}h_2\cdots g_k^{\x_k}h_k=1 \label{tame-directz-eq1}
\end{equation}
be an exponent equation of size $n$ with pairwise distinct variables
$\x_1,\ldots,\x_k$ and with $h_0,g_1,h_1,\ldots,g_k,h_k\in G\times\Z$. Let
$h_i=(\bar{h}_i,y_i)$ for $i\in[0,k]$ and $g_i=(\bar{g}_i,z_i)$ for $i\in[1,k]$.

The exponent equation
$\bar{h}_0\bar{g}_1^{\x_1}\bar{h}_1\bar{g}_2^{\x_2}\bar{h}_2\cdots
\bar{g}_k^{\x_k}\bar{h}_k=1$ has a semilinear solution set
$\bar{S}\subseteq\N^k$ of magnitude at most $\bar{p}(n)$.  The solution set of
\labelcref{tame-directz-eq1} is 
\[ S=\{ (x_1,\ldots,x_k)\in\bar{S} \mid z_1x_1+\cdots
+z_kx_k=y \}, \]
 where $y=-(y_0+\cdots +y_k)$. Note that $|z_i|\le n$ and
$|y|\le n$.  By \cref{semil-intersection}, $S$ is semilinear of magnitude
$(n^2\bar{p}(n))^d$ for some constant $d$ (recall that $k\le n$).
\end{proof}

\subsection{Tameness of free products}\label{tameness-free}
This section is devoted to the proof of the following \lcnamecref{tameness-free-preserved}.
\begin{prop}\label{tameness-free-preserved}
If $G_0$ and $G_1$ are knapsack tame, then so is $G_0*G_1$.
\end{prop}
Let $G = G_0*G_1$.
Suppose that for $i \in\{0,1\}$, the group $G_i$ is  generated by $A_i$, 
where w.l.o.g.~$A_i^{-1}=A_i$ and let $A=A_0\uplus A_1$, which generates $G$.  Recall that every $g\in G$ can be
written uniquely as $g=g_1\cdots g_n$ where $n \geq 0$, $g_i\in (G_0 \setminus\{1\}) \cup (G_1 \setminus\{1\})$ for
each $i\in[1,n]$ and where $g_j\in G_t$ iff $g_{j+1}\in G_{1-t}$ for
$j\in[1,n-1]$.  We call $g$ \emph{cyclically reduced} if for some
$t\in\{0,1\}$, either $g_1\in G_t$ and $g_n\in G_{1-t}$ or $g_1,g_n\in G_t$ and
$g_ng_1\ne 1$. Consider an exponent equation
\begin{equation}  h_0g_1^{\x_1}h_1\cdots
g_k^{\x_k}h_k=1,\label{tameness-free-expeq} \end{equation} 
of size $n$, where $g_i$ is represented by $u_i\in A^*$ for $i\in[1,k]$ and
$h_i$ is represented by $v_i\in A^*$ for $i\in[0,k]$.  Then clearly
$\sum_{i=0}^k |v_i|+\sum_{i=1}^k |u_i|\le n$. Let $S\subseteq\N^k$ be the set
of all solutions to~\labelcref{tameness-free-expeq}.  Every word $w\in A^*$ has
a (possibly empty) unique factorization into maximal factors from $A_0^+\cup
A_1^+$, which we call \emph{syllables}. By $\|w\|$, we denote the number of
syllables of $w$.  The word $w$ is \emph{reduced} if none of its syllables
represents $1$ (in $G_0$ resp. $G_1$).  We define the maps $\lambda,\rho\colon
A^+\to A^+$ (''rotate left/right''), where for each word $w\in A^+$ with its
factorization $w=w_1\cdots w_m$ into syllables, we set $\lambda(w)=w_2\cdots
w_mw_1$ and $\rho(w)=w_m w_1w_2\cdots w_{m-1}$.

Consider a word $w\in A^*$ and suppose $w=w_1\cdots w_m$, $m \geq 0$, where for each
$i\in[1,m]$, we have $w_i\in A_j^+$ for some $j\in\{0,1\}$ (we allow that $w_i, w_{i+1} \in A_j^+$). A
\emph{cancellation} is a subset $C\subseteq \Powerset{[1,m]}$ that is
\begin{itemize}
\item \emph{a partition}: $\bigcup_{I\in C} I=[1,m]$ and $I\cap J=\emptyset$ for any $I,J\in C$ with $I\ne J$.
\item \emph{consistent}: for each $I\in C$, there is an $i\in\{0,1\}$ such that
$w_j\in A_i^+$ for all $j\in I$.
\item \emph{cancelling}: if $\{i_1,\ldots,i_\ell\}\in C$ with
$i_1<\cdots<i_\ell$, then $w_{i_1}\cdots w_{i_\ell}$ represents $1$ in $G$.
\item \emph{well-nested}: there are no $I,J\in C$ with $i_1,i_2\in I$ and $j_1,j_2\in J$ such that $i_1<j_1<i_2<j_2$.
\item \emph{maximal}: if $w_i, w_{i+1} \in A_j^+$ for $j\in\{0,1\}$ then there is an $I \in C$ with $i,i+1 \in I$.
\end{itemize}
Since $C$ can be regarded as a hypergraph on $[1,m]$, the elements of $C$ will
be called \emph{edges}.  We have the following simple fact:

\begin{lem}
Let  $w=w_1\cdots w_m$, $m \geq 0$, where for each
$i\in[1,m]$, we have $w_i\in A_j^+$ for some $j\in\{0,1\}$.
Then $w$ admits a cancellation if and only if it represents $1$ in $G$.
\end{lem}

\begin{proof}
Assume that $w$ representes $1$ in the free product $G$.
The case $w = \varepsilon$ is clear; hence assume that $w \neq \varepsilon$.
 Then there must 
exist a factor $w_i w_{i+1} \cdots w_j$ representing $1$ in $G$ such that
(i) $w_i w_{i+1} \cdots w_j \in A_k^+$ for some $k \in \{0,1\}$, 
(ii) either $i=1$ or $w_{i-1} \in A_{1-k}^+$, and (iii) either $j=m$ or $w_{j+1} \in A_{1-k}^+$.
The word $w' = w_1 \cdots w_{i-1} w_{j+1}\cdots w_m$ also represents $1$ in $G$.
By induction, $w'$ admits a cancellation $C'$. Let $C''$ be obtained from $C'$ by replacing every
occurrence of an index $k \geq i$ in $C'$ by $k+j-i+1$. Then  
$C = C'' \cup \{ [i,j] \}$ is a cancellation for $w$.

For the other direction let us call a partition $C \subseteq \Powerset{[1,m]}$ a \emph{weak cancellation}
if it is consistent, cancelling and well-nested (but not necessarily maximal). Then we show by
induction that $w$ represents $1$ if it has a weak cancellation. So, let $C$ be a weak cancellation of $w \neq \varepsilon$. 
Then there must exist an interval $[i,j] \in C$ (otherwise
$C$ would be not well-nested). Then $w_i w_{i+1} \cdots w_j$ represents $1$ in $G$. Consider the word 
$w' = w_1 \cdots w_{i-1} w_{j+1}\cdots w_m$. Let $C'$ be obtained from $C \setminus \{ [i,j] \}$ by replacing
every occurrence of an index $k \geq j+1$ in $C \setminus \{ [i,j] \}$ by $k-j+i-1$.
Then $C'$ is a weak cancellation for $w'$. Hence, $w'$ represents $1$, which implies that $w$ represents $1$.
\end{proof}
Of course, when showing that the solution set of \eqref{tameness-free-expeq}
has a polynomial magnitude, we may assume that $g_i\ne 1$ for any
$i\in[1,k]$.  Moreover, we lose no generality by assuming that all words $u_i$,
$i\in [1,k]$ and $v_i$, $i\in[0,k]$ are reduced. Furthermore, we may assume
that each $g_i$ is cyclically reduced.  Indeed, if some $g_i$ is not cyclically
reduced, we can write $g_i=f^{-1}gf$ for some cyclically reduced $g$ and
replace $h_{i-1}$, $g_i$, and $h_{i}$ by $h_{i-1}f^{-1}$, $g=fg_if^{-1}$, and
$fh_{i}$, respectively.  This does not change the solution set because
$h_{i-1}f^{-1}(fg_if^{-1})^{x_i}fh_{i}=h_{i-1}g_i^{x_i}h_i$.  Moreover, if we
do this replacement for each $g_i$ that is not cyclically reduced, we increase
the size of the instance by at most $2|g_1|+\cdots+2|g_k|\le 2n$ (note that
$|g|=|g_i|$). Applying this argument again, we may even assume that 
\begin{equation}
u_i\in A_0^+\cup A_1^+\cup A_0^+A^*A_1^+ \cup A_1^+A^*A_0^+ \label{tameness-free-format}
\end{equation}
for every $i\in[1,k]$. Note that $\lambda$ and $\rho$ are bijections on words of this form.

Consider a solution $(x_1,\ldots,x_k)$ to \labelcref{tameness-free-expeq}. Then the word
\begin{equation} w=v_0u_1^{x_1}v_1\cdots u_k^{x_k}v_k \label{tameness-free-sol1}\end{equation}
 represents $1$ in $G$.  We
factorize each $v_i$, $i\in[0,k]$, and each $u_i$, $i\in[1,k]$, into its
syllables. These factorizations define a factorization $w=w_1\cdots w_m$ and we
call this the \emph{block factorization} of $w$. 
This is the
coarsest refinement of the factorization $w=v_0u_1^{x_1}v_1\cdots u_k^{x_k}v_k$
and of $w$'s factorization into syllables.
The numbers $1, 2, \ldots, m$ are the \emph{blocks} of $w$.
We fix this factorization $w = w_1 \cdots w_m$ for the rest of this section.

\subsection*{Cycles and certified solutions} In the representation $v_0u_1^{\x_1}v_1\cdots
u_k^{\x_k}v_k=1$ of \labelcref{tameness-free-expeq}, the words $u_1,\ldots,u_k$
are called the \emph{cycles}. If $u_i\in A_0^+\cup A_1^+$, the cycle $u_i$ is
said to be \emph{simple} and otherwise \emph{mixed} (note that
$u_i=\varepsilon$ cannot happen because $g_i\ne 1$). 
Let $p$ be a block of $w$. If $w_p$ is contained in 
some $u_i^{x_i}$ for a cycle $u_i$, then $p$ is a \emph{$u_i$-blocks} or
\emph{block from $u_i$}. If $w_p$ is contained in 
some $v_i$, then $p$ is a \emph{$v_i$-block}
or a \emph{block from $v_i$}.  

Note that if $C\subseteq \Powerset{[1,m]}$ is a cancellation for $w = w_1 \cdots w_m$ then
by maximality, for each simple cycle $u_i$, all $u_i$-blocks are contained in the same edge of $C$.

A \emph{certified solution} is a pair $(x,C)$, where $x$ is a solution to \labelcref{tameness-free-expeq}
and $C$ is a cancellation of the word $w$ as in \labelcref{tameness-free-sol1}.

We will also need the following two auxiliary lemmas.
\begin{lem}\label{tameness-free-mutex}
Let $C$ be a cancellation. If $i,j$ are two distinct blocks from the same
mixed cycle, then there is no edge $I\in C$ with $i,j\in I$.
\end{lem}
\begin{proof}
Suppose there is such an $I\in C$. Furthermore, assume that $i$ and $j$ are
chosen so that $|i-j|$ is minimal and $i<j$. Since $i,j\in I$, we have $w_iw_j\in A_0^+\cup A_1^+$ by consistency of $C$. Hence, by
\labelcref{tameness-free-format}, we cannot have $j=i+1$.
Therefore there is an $\ell\in [1,m]$ with $i<\ell<j$. This means there is a $J\in
C$ with $\ell\in J$. By well-nestedness, $J\subseteq [i,j]$.  Since every
edge in $C$ must contain at least two elements, we have $|J|\ge 2$ and thus a
contradiction to the minimality of $|i-j|$.
\end{proof}

An edge $I\in C$ is called \emph{standard} if
$|I|=2$ and the two blocks in $I$ are from mixed cycles. Intuitively, the
following lemma tells us that in a  cancellation, most edges are standard.
\begin{lem}\label{tameness-free-nonstandard}
Let $C$ be a  cancellation and $u_i$ be a mixed cycle. Then there are at
most $n+3k+1$ non-standard edges $I\in C$ containing a $u_i$-block.
\end{lem}
\begin{proof}
Let $N\subseteq C$ be the set of all non-standard edges $I\in C$ that contain a $u_i$-block.
Then, each edge $I\in N$ satisfies one of the following.
\begin{enumerate}[label=(\roman*)]
\item $I$ contains a block from some simple cycle. There are at
most $k$ such $I$.
\item $I$ contains a block from some $v_j$, $j\in[0,k]$. Since
$\|v_0\|+\cdots+\|v_k\|\le n$, there are at most $n$ such $I$.
\item\label{tameness-free-nonstandard-mixed} $I$ contains only blocks from mixed cycles and $|I|>2$.
\end{enumerate}
Let $M\subseteq C$ be the set of edges of type
\labelcref{tameness-free-nonstandard-mixed}. If we can show that $|M|\le 2k+1$, then the
\lcnamecref{tameness-free-nonstandard} is proven.  Consider the sets
\begin{align*}
M_- &= \{ I\in M \mid \text{$I$ contains a block from a mixed cycle $u_j$, $j<i$} \}, \\
M_+ &= \{ I\in M \mid \text{$I$ contains a block from a mixed cycle $u_j$, $j>i$} \}.
\end{align*}
We shall prove that $|M_-\cap M_+|\le 1$ and that $|M_+\setminus M_-|\le k$. By
symmetry, this also means $|M_-\setminus M_+|\le k$ and thus $|M|=|M_- \cup
M_+|\le 2k+1$.

Suppose $I_1,I_2\in M_-\cap M_+$, $I_1\ne I_2$. Let $r\in I_1$ and $s\in I_2$
such that $r$ and $s$ are blocks from $u_i$, say with $r<s$. Since $I_1\in
M_+$, $I_1$ contains a block $r'$ from a mixed cycle $u_j$, $j>i$. This
means in particular $s<r'$. By well-nestedness, this implies $I_2\subseteq
[r,r']$, so that $I_2$ cannot contain a block from a mixed cycle $u_\ell$ with
$\ell<i$, contradicting $I_2\in M_-$.  Thus, $|M_-\cap M_+|\le 1$.

In order to prove $|M_+\setminus M_-|\le k$, we need another concept.  For each
$I\in M_+$, there is a maximal $j\in [1,k]$ such that $u_j$ is a mixed cycle
and $I$ contains a block from $u_j$. Let $\mu(I)=j$. We will show
$\mu(I_1)\ne\mu(I_2)$ for all $I_1,I_2\in M_+\setminus M_-$, $I_1\ne I_2$. This
clearly implies $|M_+\setminus M_-|\le k$.

Suppose $I_1,I_2\in M_+\setminus M_-$, $I_1\ne I_2$, with $\mu(I_1)=\mu(I_2)$.
Let $j=\mu(I_1)=\mu(I_2)$. Let $r$ be a block from $u_i$ contained in $I_1$
and let $r'$ be a block from $u_i$ contained in $I_2$. (Recall that those
exist because $I_1,I_2\in M$.) Without loss of generality, assume $r<r'$.
Moreover, let $s$ be a block from $u_j$ contained in $I_1$ and let $s'$
be a block from $u_j$ contained in $I_2$.  Thus, we have $r < r' < s'$.

However, we have $|I_1|>2$, meaning $I_1$ contains a block $p$ other than
$r$ and $s$.  Since an edge cannot contain two blocks of one mixed cycle
(\cref{tameness-free-mutex}), $p$ has to belong to a mixed cycle $u_t$
other than $u_i$ and $u_j$.  By the maximality of $j$, we have $i<t<j$. This
implies, however, $r < r'<p<s'$, which contradicts well-nestedness.
\end{proof}

\subsection*{Mixed periods}
From now on, for each $i\in[1,k]$, we use $e_i$ to denote the $i$-th unit
vector in $\N^k$, i.e. the vector with $1$ in the $i$-th coordinate and $0$
otherwise.  A \emph{mixed period} is a vector $\pi\in\N^k$ of the form
$\|u_j\|\cdot e_i + \|u_i\|\cdot e_j$, where $u_i$ and $u_j$ are mixed cycles.
Let $\Periods\subseteq\N^k$ be the set of
mixed periods.  Note that $|\Periods|\le k^2$. 

We will need a condition that guarantees that a given period $\pi\in \Periods$ can be
added to a solution $x$ to obtain another solution.  Suppose we have two blocks
$p$ and $q$ for which we know that if we insert a string $f_1$ to the left
of $w_p$ and a string $f_2$ to the right of $w_q$ and $f_1f_2$ cancels to $1$
in $G$, then the whole word cancels to $1$. Which string would we insert to the
left of $w_p$ and to the right of $w_q$ if we build the solution $x+\pi$?

Suppose $p$ is a $u_i$-block and $q$ is a $u_j$-block. Moreover, let $r$
be the first (left-most) $u_i$-block and let $s$ be the last (right-most)
$u_j$-block.  If we add $\|u_j\|\cdot e_i$ to $x$, this inserts
$\lambda^{p-r}(u_i^{\|u_j\|})$ to the left of $w_p$: Indeed, in the case $p=r$,
we insert $u_i^{\|u_j\|}$; and when $p$ moves one position to the right, the
inserted string is rotated once to the left. Similarly, if we add
$\|u_i\|\cdot e_j$ to $x$, we insert $\rho^{s-q}(u_j^{\|u_i\|})$ to the right of $w_q$:
This is clear for $q=s$ and decrementing $q$ means rotating the inserted string
to the right. This motivates the following definition.

\label{tameness-free-compatibility}
Let $(x,C)$ be a certified solution and let $u_i$ and $u_j$ be mixed cycles with $i<j$.
Moreover, let $r \in [1,m]$ be the left-most $u_i$-block and let $s \in [1,m]$ be the
right-most $u_j$-block. Then the mixed period $\pi=\|u_j\|\cdot e_i + \|u_i\|\cdot e_j$ is
\emph{compatible with $(x,C)$} if there are a $u_i$-block $p$ and a $u_j$-block $q$ 
such that
\begin{align}
\{p,q\}\in C \text{ and $\lambda^{p-r}(u_i^{\|u_j\|})\rho^{s-q}(u_j^{\|u_i\|})$ represents $1$ in $G$} . \label{tameness-free-compatibility-cond}
\end{align}
With $\CompPeriods(x,C)$, we denote the set of mixed periods that are
compatible with $(x,C)$.  One might wonder why we require an edge $\{p,q\}\in C$.
In order to guarantee that $\lambda^{p-r}(u_i^{\|u_j\|})$ and
$\rho^{s-q}(u_j^{\|u_i\|})$ can cancel, it would be sufficient to merely forbid
edges $I\in C$ that intersect $[p,q]$ and contain a block outside
of $[p-1,q+1]$. However, this weaker condition can become false when we insert
other mixed periods. Our stronger condition is preserved, which implies:

\begin{lem}\label{tameness-free-periods}
Let $(x,C)$ be a certified solution. Then every 
$x'\in x+\CompPeriods(x,C)^\oplus$ is a solution.
\end{lem}
\begin{proof}
It suffices to show that if $(x,C)$ is a certified solution and
$\pi\in\CompPeriods(x,C)$, then there is a certified solution $(x',C')$ such
that $x'=x+\pi$ and $\CompPeriods(x,C)\subseteq \CompPeriods(x',C')$.  Suppose
$\pi=\|u_j\|\cdot e_i + \|u_i\|\cdot e_j \in\CompPeriods(x,C)$. Without loss of
generality, assume $i<j$. Let $r\in[1,m]$ be the left-most
$u_i$-block and $s\in[1,m]$ be the right-most $u_j$-block in $w$. 
Since $\pi\in\CompPeriods(x,C)$, there is a $u_i$-block $p$ and a $u_j$-block
$q$ such that \labelcref{tameness-free-compatibility-cond} holds. As
explained above, we can insert $\lambda^{p-r}(u_i^{\|u_j\|})$ on the left of
$w_p$ and $\rho^{s-q}(u_j^{\|u_i\|})$ on the right of $w_q$ and thus obtain a
word $w'$ that corresponds to the vector $x'=x+\pi$. 

Both inserted words consist of
$\|u_j\|\cdot\|u_i\|$ many blocks and they cancel to $1$, which means we can
construct a cancellation $C'$ from $C$ as follows. Between the two sequences of
inserted blocks, we add two-element edges so that the left-most inserted
$u_i$-block is connected to the right-most inserted $u_j$-block, and so forth.
The blocks that existed before are connected by edges as in $C$. It is clear
that then, $C'$ is a partition that is consistent, cancelling and maximal. Moreover,
since there is an edge $\{p,q\}$, the new edges between the inserted blocks do
not violate well-nestedness: If there were a crossing edge, then there would
have been one that crosses $\{p,q\}$.  

It remains to verify $\CompPeriods(x,C)\subseteq \CompPeriods(x',C')$. 
A mixed period $\pi' \in \CompPeriods(x,C) \setminus \{\pi\}$ is clearly contained
in $\CompPeriods(x',C')$ too. Hence, it remains to show $\pi \in  \CompPeriods(x',C')$.
This, however, follows from the fact that instead of the edge $\{p,q\}$ that witnesses
compatibility of $\pi$ with $(x,C)$, we can use its counterpart in $C'$; let us call this edge $\{p',q'\}$: 
If $r'$ is the left-most 
$u_i$-block in $w'$ and $s'$ is the right-most $u_j$-block in $w'$, then $p'-r' = p-r + \|u_i\| \cdot \|u_j\|$ and
$s'-q' = s-q + \|u_i\| \cdot \|u_j\|$. This implies 
$\lambda^{p-r}(u_i^{\|u_j\|})=\lambda^{p'-r'}(u_i^{\|u_j\|})$ and 
$\rho^{s-q}(u_j^{\|u_i\|})=\rho^{s'-q'}(u_j^{\|u_i\|})$ which implies that 
\labelcref{tameness-free-compatibility-cond} holds for the edge $\{p',q'\}$. This
completes the proof of the \lcnamecref{tameness-free-periods}.
\end{proof}

We shall need another auxiliary \lcnamecref{tameness-free-consecutive}.
\begin{lem}\label{tameness-free-consecutive}
Let $C$ be a cancellation for $w$. Let $u_i$ and $u_j$ be distinct mixed cycles.  Let
$D\subseteq C$ be the set of standard edges $I\in C$ that contain one block from
$u_i$ and one block from $u_j$.  Then the set 
$B = \{ p \in [1,m] \mid p \text{ is a $u_i$-block}, \exists I \in D : p \in I\}$
is an interval.
\end{lem}
\begin{proof}
We prove the case $i<j$, the other follows by symmetry.  Suppose there are
$r_1,r_2\in B$ such that $r_1<r_2$ and there is no $t\in B$ with $r_1<t<r_2$.

Towards a contradiction, suppose $r_2-r_1>1$.  Since $r_1,r_2\in B$, there
are $I_1,I_2\in D$ with $r_1\in I_1$ and $r_2\in I_2$.  Let $I_1=\{r_1,s_1\}$
and $I_2=\{r_2,s_2\}$. Then $s_1$ and $s_2$ are $u_j$-blocks and by
well-nestedness, we have $r_1 < r_2 < s_2 < s_1$. Since $r_2-r_1>1$, there is
a $t$ with $r_1<t<r_2$ and therefore some $J\in C$ with $t\in J$. Since $|J|\ge
2$, there has to be a $t'\in J$, $t'\ne t$. However, well-nestedness dictates
that $t'\in[r_1,s_1]\setminus [r_2,s_2]$.  Since $J$ cannot contain another
block from $u_i$ (\cref{tameness-free-mutex}), we cannot have $t'\in
[r_1,r_2]$, which only leaves $t'\in[s_2,s_1]$. Hence, $t'$ is from
$u_j$.  By the same argument, any block $t''\in J\setminus\{t,t'\}$
must be from $u_i$ or $u_j$, contradicting \cref{tameness-free-mutex}. This
means $|J|=2$ and thus $t\in B$, in contradiction to the choice of $r_1$ and $r_2$.
\end{proof}

Let $M\subseteq [1,k]$ be the set of $i\in[1,k]$ such that $u_i$ is a mixed
cycle. We define a new norm on vectors $x\in\N^k$ by setting $\mixedNorm{x}=\max_{i\in M} x_i$.
\begin{lem}\label{tameness-free-removeperiods}
There is a polynomial $q$ such that the following holds.  For every certified
solution $(x,C)$ with $\mixedNorm{x}>q(n)$, there exists a mixed period $\pi\in
\Periods(x,C)$ and a certified solution $(x',C')$ such that $x'=x-\pi$
and $\CompPeriods(x,C)\subseteq\CompPeriods(x',C')$.
\end{lem}
\begin{proof}
We show that the \lcnamecref{tameness-free-removeperiods} holds if $q(n)\ge
(n+3k+1) + kn^2$.  (Recall that $k\le n$.) Let $(x,C)$ be a certified
solution with $\mixedNorm{x}>q(n)$.  Then there is a mixed cycle $u_i$ such
that $x_i>q(n)$ and hence $u_i^{x_i}$ consists of more than $q(n)$ blocks. Let
$D\subseteq C$ be the set of all edges $I\in C$ that contain a block
from $u_i$.  Since an edge can contain at most one block per mixed cycle 
(\cref{tameness-free-mutex}), we have $|D|>q(n)$.
Hence, \cref{tameness-free-nonstandard} tells us
that $D$ contains more than $k n^2$ standard edges. Hence, there exists a
mixed cycle $u_j$ such that the set $E\subseteq D$ of standard edges $I\in D$ that 
consist of one block from $u_i$ and one block from $u_j$
satisfies $|E| > n^2$.  If $B_i$
(resp., $B_j$) denotes the set of blocks from $u_i$ (resp., $u_j$) contained in some edge $I\in
E$, then each of the sets $B_i$ and $B_j$
has to be an interval (\cref{tameness-free-consecutive}) of size more than $n^2$.

We only deal with the case $i<j$, the case $i>j$ can be done similarly.
Let us take a subinterval $[p',p]$ of $B_i$ such that $p-p' =  \|u_i\|\cdot \|u_j\| \leq n^2$.
 By well-nestedness and since $B_j$ is an interval, the neighbors (with respect to the edges from $E$) of $[p',p]$ form an interval $[q,q'] \subseteq B_j$
as well, and we have $p-p'=q'-q=\|u_i\|\cdot \|u_j\|$.
Moreover, we have an edge $\{p-\ell, q+\ell\} \in E$ for each
$\ell\in[0,p-p']$.  In particular, $w_{p'}w_{p'+1}\cdots w_{p-1}w_{q+1}\cdots
w_{q'}$ represents $1$ in $G$.

Let $r$ be the left-most $u_i$-block and let $s$ be the right-most
$u_j$-block.  Then, as shown before the definition of compatibility
(p.~\pageref{tameness-free-compatibility}), we have
\begin{align*}
\lambda^{p-r}(u_i^{\|u_j\|}) = w_{p'}w_{p'+1}\cdots w_{p-1}, && \rho^{s-q}(u_j^{\|u_i\|}) = w_{q+1}w_{q+1}\cdots w_{q'}.
\end{align*}
Therefore, 
$\lambda^{p-r}(u_i^{\|u_j\|})\rho^{s-q}(u_j^{\|u_i\|})$
represents $1$ in $G$ and $\{p,q\}$ witnesses compatibility of
$\pi=\|u_j\|\cdot e_i+\|u_i\|\cdot e_j$ with $(x,C)$. Hence,
$\pi\in\CompPeriods(x,C)$.

Let $x'=x-\pi$. We remove the factors $w_{p'} \cdots w_{p-1}$ and
$w_{q+1}\cdots w_{q'}$ from $w$. Then, the remaining blocks spell
$w'=v_0u_1^{x'_1}v_1\cdots u_k^{x'_k}v_k$.
Indeed, recall that removing from a word $y^t$ any factor of length $\ell\cdot
|y|$ will result in the word $y^{t-\ell}$. Moreover, let $C'$ be the set of
edges that agree with $C$ on the remaining blocks. By the choice of the
removed blocks, it is clear that $C'$ is a cancellation for $w'$. Hence,
$(x',C')$ is a certified solution.

It remains to verify $\CompPeriods(x,C)\subseteq\CompPeriods(x',C')$.  
First note that for every mixed cycle $u_\ell$, all 
$u_\ell$-blocks that remain in $w'$ change their position relative to the left-most and the
right-most $u_\ell$-block by a difference that is divisible by $\|u_\ell\|$ (if $i \neq \ell \neq j$ 
then these relative positions do not change at all). 
Note that the expression $\lambda^{p-r}(u_i^{\|u_j\|})$ is not altered when
$p-r$ changes by a difference divisible by $\|u_i\|$, and an analogous
fact holds for $\rho^{s-q}(u_j^{\|u_i\|})$.  
Hence, the edge in $C'$ that corresponds
to the $C$-edge $\{p,q\}$ is a witness for  $\pi \in \CompPeriods(x',C')$.
Moreover, for all other mixed periods $\pi' \in \CompPeriods(x,C) \setminus \{\pi\}$ that are witnessed
by an edge $\{t,u\} \in C$, the blocks $t$ and $u$ do not belong to $[p',p-1] \cup [q+1,q']$.
Therefore, the corresponding edge in $C'$ exists and serves as a witness for 
$\pi' \in \CompPeriods(x',C')$. 
\end{proof}

Repeated application of \cref{tameness-free-periods} now yields:
\begin{lem}\label{tameness-free-mixed}
There exists a polynomial $q$ such that the following holds.
For every solution $x\in\N^k$, there exists a certified solution
$(x',C')$ such that $\mixedNorm{x'} \leq q(n)$ and $x\in
x'+\CompPeriods(x',C')^\oplus$.
\end{lem}
\begin{proof}
Let $q$ be the polynomial provided by \cref{tameness-free-removeperiods}.
Since $x$ is a solution, there is a certified solution $(x,C)$.  Repeated
application of \cref{tameness-free-removeperiods} yields certified solutions
$(x_0,C_0),\ldots,(x_m,C_m)$ and mixed periods $\pi_1,\ldots,\pi_m$ such that
$(x_0,C_0)=(x,C)$, $\pi_i\in\CompPeriods(x_{i-1},C_{i-1}) \subseteq\CompPeriods(x_i,C_i)$,
$x_{i}=x_{i-1}-\pi_i$, and $\mixedNorm{x_m} \leq q(n)$. In
particular, $\CompPeriods(x_m,C_m)$ contains each $\pi_i$ and hence
\[x=x_m+\pi_1+\cdots+\pi_m\in x_m+\CompPeriods(x_m,C_m)^\oplus. \]
Thus, $(x',C')=(x_m,C_m)$ is the desired certified solution.
\end{proof}

We are now ready to prove \cref{tameness-free-preserved} and thus \cref{tforest-tame}.
\begin{proof}[Proof of \cref{tameness-free-preserved}]
Suppose that $p_0$ and $p_1$ are the polynomials guaranteed by the knapsack
tameness of $G_0$ and $G_1$, respectively.  Recall that $S\subseteq\N^k$ is the
set of solutions to~\labelcref{tameness-free-expeq}.  We prove that there exists a
polynomial $p$ such that for every $x \in S$
there is a semilinear set $S'\subseteq\N^k$ of
magnitude at most $p(n)$ such that $x\in S'\subseteq S$.  This clearly implies
that $S$ has magnitude at most $p(n)$.
First, we apply \cref{tameness-free-mixed}. It yields a polynomial $q$ and a
certified solution $(x',C')$ with $\mixedNorm{x'}\le q(n)$ such that $x\in
x'+\CompPeriods(x',C')^\oplus$.  Let
$w' = v_0u_1^{x'_1}v_1\cdots u_k^{x'_k}v_k$
and consider $w'$ decomposed into blocks as we did above with $w$.

Let $T\subseteq [1,k]$ be the set of all $i\in[1,k]$ for which the cycle
$u_i$ is simple. Since $C'$ is maximal, for each $i\in T$, all $u_i$-blocks
are contained in one edge $I_i\in C'$. Note that it is allowed that one edge
contains the blocks of multiple simple cycles.  We partition $ T$ into
sets $ T= T_1\uplus\cdots\uplus T_t$ so that $i\in T$ and
$j\in T$ belong to the same part if and only if the $u_i$-blocks and the
$u_j$-blocks belong to the same edge of $C$, i.e. $I_i=I_j$.

For a moment, let us fix an $\ell\in[1,t]$ and let $I\in C'$ be the edge
containing all $u_i$-blocks for all the $i\in  T_\ell$. Moreover, let
$ T_\ell=\{i_1,\ldots,i_r\}$. The words $\bar{v}_j$ for $j\in[0,r]$ will
collect those blocks that belong to $I$ but are not $u_{i_s}$-blocks for any
$s\in[1,r]$.  Formally, $\bar{v}_0$ consists of all blocks that belong to $I$
that are to the left of all $u_{i_1}$-blocks. Similarly, $\bar{v}_r$ is the
concatenation of all blocks belonging to $I$ that are to the right of all
$u_{i_r}$-blocks. Finally, for $j\in[1,r-1]$, $\bar{v}_j$ consists of all
blocks that belong to $I$ and are to the right of all $u_{i_j}$-blocks and to
the left of all $u_{i_{j+1}}$-blocks.
By consistency of $C'$, for some $s\in\{0,1\}$, all the words $\bar{v}_j$ for
$j\in[0,r]$ and the words $u_{i_j}$ for $j\in [1,r]$ belong to $A_s^*$ and thus
represent elements of $G_s$. Since $G_s$ is knapsack tame, we know that the set
\[ S_\ell=\{ z\in\N^k \mid \bar{v}_0 u_{i_1}^{z_{i_1}} \bar{v}_1 u_{i_2}^{z_{i_2}} \bar{v}_2\cdots u_{i_r}^{z_{i_r}}\bar{v}_r~\text{represents $1$ in $G_s$,~~ $z_j=0$ for $j\notin  T_\ell$}\}\] 
has magnitude at most $p_s(n)$. Consider the
vector $y\in\N^k$ with $y_i=0$ for $i\in T$ and $y_i=x'_i$ for $i\in
[1,k]\setminus T$ (i.e. when $u_i$ is a mixed cycle). We claim that
$S' = y + S_1 + \cdots S_t + \CompPeriods(x',C')^\oplus$
has magnitude at most $q(n)+p_0(n)+p_1(n)+n$ and satisfies $x\in S'\subseteq
S$.  

First, since $y$ and the members of $S_1,\ldots,S_t$ are non-zero on pairwise
disjoint coordinates, the magnitude of $y+S_1+\cdots+S_t$ is the maximum of
$\|y\|$ and the maximal magnitude of $S_1,\ldots,S_\ell$.  Hence, it is
bounded by $q(n)+p_0(n)+p_1(n)$. The summand $\CompPeriods(x',C')^\oplus$ contributes
only periods, and their magnitude is bounded by $n$ (recall that they are mixed
periods). Thus, the magnitude of $S'$ is at most $p(n)=q(n)+p_0(n)+p_1(n)+n$.

The cancelling property of $(x',C')$ tells us that $x'-y$ is contained in
$S_1+\cdots+S_t$.  By the choice of $(x',C')$, we have $x\in
x'+\CompPeriods(x',C')^\oplus$.  Together, this means $x\in S'$. Hence, it
remains to show $S'\subseteq S$. To this end, consider a vector $x''\in
y+S_1+\cdots+S_t$.  It differs from $x'$ only in the exponents at simple
cycles. Therefore, we can apply essentially the same cancellation to $x''$ as
to $x'$: we just need to adjust the edges containing the blocks of simple
cycles. It is therefore clear that the resulting cancellation $C''$ has the
same compatible mixed periods as $C'$: $\CompPeriods(x'',C'')=\CompPeriods(x',C')$.
Thus, by \cref{tameness-free-periods}, we have $x''+\CompPeriods(x',C')^\oplus
\subseteq S$. This proves $S' = y+S_1+\cdots+S_t+\CompPeriods(x',C')^\oplus\subseteq
S$ and hence the \lcnamecref{tameness-free-preserved}.
\end{proof}

\subsection{$\LogCFL$-membership}
In this section, we prove the upper bound in \cref{logcfl-completeness}:

\begin{prop}\label{tforest-logcfl-membership}
Let $G=\dG(A,I)$ be a graph group where $(A,I)$ is a transitive forest. Then
the knapsack problem and the subset sum problem for $G$ belong to $\LogCFL$.
\end{prop}
\begin{proof}
Let us consider knapsack.
According to \cref{thm-acyclic-logcfl}, it suffices to provide a logspace reduction from the
knapsack problem over $G$ to the membership problem for acyclic automata over
$G$.  Suppose we have an instance $h_0g_1^{\x_1}h_1\cdots g_k^{\x_k} h_k=1$ of
the knapsack problem over $G$ of size $n$. Moreover, let $h_i$ be represented
by $v_i\in A^*$ for each $i\in[0,k]$ and let $g_i$ be represented by $u_i\in
A^*$ for $i\in[1,k]$.

By \cref{tforest-tame}, there is a polynomial $p$ such that the
above instance has a solution if and only if it has a solution $x\in\N^k$ with
$\|x\| \le p(n)$. We construct an acyclic automaton
$\cA=(Q,A,\Delta,q_0,q_f)$ as follows. It has the state set $Q=[0,k+1]\times
[0,p(n)]$ and the following transitions. From $(0,0)$, there is one transition
labeled $v_0$ to $(1,0)$. For each $i\in[1,k]$ and $j\in[0,p(n)-1]$, there are
two transitions from $(i,j)$ to $(i,j+1)$; one labeled by $u_i$ and one labeled
by $\varepsilon$. Furthermore, there is a transition from $(i,p(n))$ to
$(i+1,0)$ labeled $v_i$ for each $i\in[1,k]$. The initial state is $q_0=(0,0)$
and the final state is $q_f=(k+1,0)$.  

It is clear that $\cA$ accepts a word that represents $1$ if and only if the
exponent equation has a solution.  Finally, the reduction can clearly be
carried out in logarithmic space. 

For subset sum the same reduction as above works but the polynomial bound on solutions is for free.
\end{proof}

\subsection{$\LogCFL$-hardness}

It remains to show the lower bound
in \cref{logcfl-completeness}.  If $(A,I)$ is not complete, then $(A,I)$ contains two
non-adjacent vertices and thus $\dG(A,I)$ contains an isomorphic copy of $F_2$,
the free group of rank two. Hence, we will show that knapsack and subset sum for $F_2$ are
$\LogCFL$-hard.  Let $\{a,b\}$ be a generating set for $F_2$.  Let $\theta
\colon \{a,b,a^{-1},b^{-1}\}^* \to F_2$ be the morphism that maps a word $w$ to
the group element represented by $w$. 

A \emph{valence automaton} over a group
$G$ is a tuple $\cA=(Q,\Sigma,\Delta,q_0,q_f)$ where $Q$, $\Sigma$, $q_0$, $q_f$
are as in a finite automaton and $\Delta$ is a finite subset of
$Q\times\Sigma^*\times G\times Q$. The \emph{language accepted by $\cA$} is denoted $L(\cA)$ and consists of
all words $w_1\cdots w_n$ such that there is a computation
$p_0\xrightarrow{w_1,g_1} p_1 \to\cdots\to p_{n-1}\xrightarrow{w_n,g_n} p_n$
such that $(p_{i-1}, w_i, g_i, p_i)\in\Delta$ for $i\in[1,n]$, $p_0=q_0$, $p_n=q_f$, and $g_1\cdots g_n=1$ in $G$.
We call this computation also an {\em accepting run} of $\mathcal{A}$ for $w$ (of length $n$). Note that we allow $\varepsilon$-transitions of the form
$(p,\varepsilon,g,q) \in \Delta$. This implies that an accepting run for a word $w$ can be of length 
greater than $|w|$.

An analysis of a proof (in this case \cite{Kambites2009}) of the Chomsky-Sch\"{u}tzenberger theorem yields:
\begin{lem} \label{lemma-kambites}
For every  language $L \subseteq \Sigma^*$ 
the following statements are equivalent:
\begin{enumerate}[label=(\roman*)]
\item $L$ is context-free.
\item There is a valence automaton $\mathcal{A}$ over $F_2$ such that $L = L(\mathcal{A})$.
\item There is a valence automaton $\mathcal{A}$ over $F_2$ and a constant $c \in \N$ such that $L = L(\mathcal{A})$ and
for every $w \in L$ there exists an accepting run of $\mathcal{A}$ for $w$ of length at most $c \cdot |w|$.
\end{enumerate}
\end{lem}

\begin{proof}
\newcommand{\valedge}[2]{(#1,#2)}
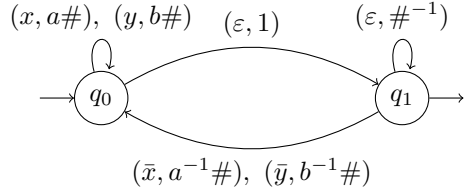
\begin{figure}
\centering
\begin{tikzpicture}[every state/.style={minimum size=10pt}]
\node[state, initial, initial text=]  (q0) at (0,0) {$q_0$};
\node[state,accepting by arrow]       (q1) at (4,0) {$q_1$};
\path[->] (q0) edge [loop above] node {$\valedge{x}{a\#},~\valedge{y}{b\#}$} (q0)
          (q1) edge [loop above] node {$\valedge{\varepsilon}{\#^{-1}}$} (q1)
	  (q0) [bend left] edge [above] node {$\valedge{\varepsilon}{1}$} (q1)
	  (q1) edge [below] node {$\valedge{\bar{x}}{a^{-1}\#},~\valedge{\bar{y}}{b^{-1}\#}$} (q0);
\end{tikzpicture}
\caption{Transducer used in proof of \cref{lemma-kambites}}
\label{transducer-kambites}
\end{figure}
The equivalence of (i) and (ii) is well known (see \cite{Kambites2009}) and the implication from (iii) to (ii) is trivial.
We show that (i) implies (iii). For this, 
we shall use the concept of rational transductions.  If $\Sigma$ and $\Gamma$
are alphabets, subsets $T\subseteq\Gamma^*\times\Sigma^*$ are called
\emph{transductions}.  Given a language $L\subseteq\Sigma^*$ and a transduction
$T\subseteq \Gamma^*\times\Sigma^*$, we define
\[ TL=\{u\in \Gamma^* \mid \text{$(u,v)\in T$ for some $v\in L$} \}. \]
A \emph{finite-state transducer} is a tuple
$\cA=(Q,\Sigma,\Gamma,\Delta,q_0,q_f)$, where $Q$ is a finite set of
\emph{states}, $\Sigma$ is its \emph{input alphabet}, $\Gamma$ is its
\emph{output alphabet}, $\Delta$ is a finite subset of
$Q\times\Gamma^*\times\Sigma^*\times Q$, $q_0\in Q$ is its \emph{initial
state}, and $q_f\in Q$ is its \emph{final state}. The elements of $\Delta$ are
called \emph{transitions}. We say that a pair $(u,v)\in\Gamma^*\times\Sigma^*$
is \emph{accepted by $\cA$} if there is a sequence
$(p_0,u_1,v_1,p_1),(p_1,u_2,v_2,p_2),\ldots,(p_{n-1},u_n,v_n,p_n)$ of
transitions where $n\ge 1$, $p_0=q_0$, $p_n=q_f$, $u=u_1\cdots u_n$, and
$v=v_1\cdots v_n$.  The set of all pairs $(u,v)\in\Gamma^*\times\Sigma^*$ that
are accepted by $\cA$ is denoted by $T(\cA)$. A transduction $T\subseteq
\Gamma^*\times\Sigma^*$ is called \emph{rational} if there is a finite-state
transducer $\cA$ with $T(\cA)=T$. 

Let $W_2\subseteq \{a,b,a^{-1},b^{-1}\}^*$ be the word problem of $F_2$, i.e.
\[ W_2=\{w\in \{a,b,a^{-1},b^{-1}\}^* \mid \theta(w)=1\}. \]

For languages $K\subseteq \Gamma^*$ and $L\subseteq\Sigma^*$, we write
$K\leadsto L$ if there is a rational transduction $T \subseteq \Gamma^* \times \Sigma^*$ and a constant $c$ such
that $K=TL$ and for each $u\in K$, there is a $v\in L$ with $|v|\le c|u|$ and
$(u,v)\in T$. Observe that the relation $\leadsto$ is transitive, meaning that it suffices
to show $L\leadsto W_2$ for every context-free language $L$.

Let $D_2$ be the one-sided Dyck language over two pairs of parentheses, in
other words: $D_2$ is the smallest language $D_2\subseteq
\{x,\bar{x},y,\bar{y}\}^*$ such that $\varepsilon\in D_2$ and whenever $uv\in
D_2$, we also have $uwv\in D_2$ for $w\in\{x\bar{x}, y\bar{y}\}$.

It is easy to see that $L\leadsto D_2$ for every context-free language $L$.
Indeed, an $\varepsilon$-free pushdown automaton (which exists for every
context-free language~\cite{Greibach1965}) for $L$ can be converted into a transducer
witnessing $L\leadsto D_2$. Therefore, it remains to show that $D_2\leadsto W_2$.

Let $F_3$ be the free group of rank $3$ and let $\{a,b,\#\}$ be a free
generating set for $F_3$. As above, let 
\[ W_3=\{w\in \{a,b,\#,a^{-1},b^{-1},\#^{-1}\}^* \mid \text{$w$ represents $1$ in $F_3$}\} \]
 be the word problem of $F_3$. Since $F_3$ can be embedded into
$F_2$~\cite[Proposition~3.1]{LySch77}, we clearly have $W_3\leadsto W_2$. It therefore suffices to
show $D_2\leadsto W_3$.

For this, we use a construction of Kambites~\cite{Kambites2009}.  He proves
that if $\cA$ is the transducer in \cref{transducer-kambites} and $T=T(\cA)$,
then $D_2=TW_3$.  Thus, for every $u\in D_2$, we have $(u,v)\in T$ for some
$v\in W_3$.  An inspection of $\cA$ yields that $|v|=2|u|+|v|_{\#^{-1}}$ and
$|v|_{\#}=|u|$. Since $v\in W_3$, we have $|v|_{\#^{-1}}=|v|_{\#}$ and thus
$|v|=3|u|$. Hence, the transduction $T$ witnesses $D_2\leadsto W_3$.  We have
thus shown $L\leadsto D_2\leadsto W_3\leadsto W_2$ and hence the
\lcnamecref{lemma-kambites}.
\end{proof}

Given $w$, it is easy to convert the valence automaton $\cA$ from
\cref{lemma-kambites} into an acyclic automaton that exhausts all computations
of $\cA$ of length $c\cdot|w|$. This yields the following.
\begin{prop}\label{acyclic-logcfl-hard}
For $F_2$, the membership problem for acyclic automata is \LogCFL-hard.
\end{prop}
\begin{proof}
Fix a context-free language $L \subseteq \Sigma^*$ with a 
\LogCFL-complete membership problem; such languages exist \cite{Gre73}.
Fix a valence automaton $\cA=(Q,\Sigma,\Delta,q_0,q_f)$ over $F_2$ and a constant $c \in \mathbb{N}$ such that the statement
of \cref{lemma-kambites}(iii) holds for $L$, $\mathcal{A}$, and $c$.
Consider a word $w \in \Sigma^*$. From $w$ we construct an acyclic automaton $\mathcal{B}$ over 
the input alphabet $\{a,b,a^{-1},b^{-1}\}$ such that $1 \in \theta(L(\mathcal{B}))$ if and only if $w \in L$.
Let $m = |w|$, $w = a_1 a_2 \cdots a_m$ and $n = c \cdot m$.
The set of states of $\mathcal{B}$ is $[0,m] \times [0,n] \times Q$.
The transitions of $\mathcal{B}$ are defined as follows:
\begin{itemize}
\item $(i-1,j-1,p) \xrightarrow{\ x \ } (i,j,q)$ if $(p,a_i, x, q) \in \Delta$ for all $i \in [1,m]$, $j \in [1,n]$, and $x \in \{a,b,a^{-1},b^{-1}\}^*$
\item $(i,j-1,p) \xrightarrow{\ x \ } (i,j,q)$ if $(p,\varepsilon, x, q) \in \Delta$ for all $i \in [0,m]$, $j \in [1,n]$, and $x \in \{a,b,a^{-1},b^{-1}\}^*$
\end{itemize}
The initial state of $\mathcal{B}$ is $(0,0,q_0)$ and all states
$(m,j,q_f)$ with $j \in [0,n]$ are final in $\mathcal{B}$. It is then straightforward to show that 
$1 \in \theta(L(\mathcal{B}))$ if and only if $w \in L$. The intuitive idea is that in a state of $\mathcal{B}$ we store in the first component
the current position in the word $w$. In this way we enforce the simulation of a run of $\mathcal{A}$ on input $w$. In the second
component of the state we store the total number of simulated $\mathcal{A}$-transitions. In this way we make $\mathcal{B}$ acyclic.
Finally, the third state component of $\mathcal{B}$ stores the current $\mathcal{A}$-state.
\end{proof}

\begin{prop}
For $F_2$, knapsack and subset sum are \LogCFL-hard.
\end{prop}
\begin{proof}
Let $\mathcal{A} = (Q,\{a,b,a^{-1},b^{-1}\}, \Delta, q_0, q_f)$ be an acyclic automaton. We construct words $w, w_1, \ldots, w_m \in \{a,b,a^{-1},b^{-1}\}$
such that the following three statements are equivalent:
\begin{enumerate}[label=(\roman*)]
\item  $1 \in \theta(L(\mathcal{A})$.
\item $\theta(w) \in \theta(w_1^* w_2^* \cdots w_m^*)$.
\item $\theta(w) \in \theta(w_1^{e_1} w_2^{e_2} \cdots w_m^{e_m})$   for some $e_1, e_2, \ldots, e_m \in \{0,1\}$.
\end{enumerate}
W.l.o.g. assume that $Q = \{1,\ldots, n\}$, where $1$ is the initial state and $n$ is the unique final state of $\mathcal{A}$.

 Let $\alpha_i = a^i b a^{-i}$ for $i\in[1,n+2]$. It is well known that
the $\alpha_i$ generate a free subgroup of rank $n+2$ in $F_2$ \cite[Proposition~3.1]{LySch77}. Define the embedding $\varphi \colon F_2 \to F_2$
by $\varphi(a) = \alpha_{n+1}$ and $\varphi(b) = \alpha_{n+2}$. For a transition
$t = (p,w,q) \in \Delta$ let $\tilde{t} =  \alpha_p \varphi(w) \alpha_q^{-1}$.  Let $\Delta=\{t_1, \ldots, t_m\}$ such that 
$t_i = (p,a,q)$ and $t_j = (q,b,r)$ implies $i < j$. Since $\mathcal{A}$ is acyclic, such an enumeration must exist.
Together with the fact that the $\alpha_i$ generate a free group, it follows that the following three statements are equivalent:
\begin{enumerate}[label=(\roman*)]
\item  $1 \in \theta(L(\mathcal{A})$.
\item $\theta(\alpha_1 \alpha_n^{-1}) \in \theta(\tilde{t}_1^* \, \tilde{t}_2^* \cdots \tilde{t}_m^*)$.
\item $\theta(\alpha_1 \alpha_n^{-1}) \in \theta(\tilde{t}_1^{e_1} \,\tilde{t}_2^{e_2} \cdots \tilde{t}_m^{e_m})$   for some $e_1, e_2, \ldots, e_m \in \{0,1\}$.
\end{enumerate}
This shows the proposition.
\end{proof}

\section{$\NP$-completeness}

In \cite{LohreyZ16}, the authors proved that knapsack for the graph group $\dG(\mathsf{C4}) \cong F_2 \times F_2$ is {\sf NP}-complete.
Here we extend this result to all graph groups $\dG(A,I)$ where $(A,I)$ is not a transitive forest. For this, it suffices to show that 
knapsack for $\dG(\mathsf{P4})$ is {\sf NP}-complete. We take the copy $(\{a,b,c,d\}, I)$ of $\mathsf{P4}$ shown in Figure~\ref{fig:P4-C4}.

An {\em acyclic loop automaton} is a finite automaton $\mathcal{A} =
(Q,\Sigma, \Delta, q_0, q_f)$ such that there exists a linear order $\preceq$
on $\Delta$ having the property that for all $(p,u,q), (q,v,r) \in \Delta$ it
holds $(p,u,q) \preceq (q,v,r)$. Thus, an acyclic loop automaton is obtained
from an acyclic automaton by attaching to some of the states a unique loop.
For a trace monoid $\dM(A,I)$, the  \emph{intersection nonemptiness problem for acyclic loop automata}
is the following computational problem:
 
\smallskip
\noindent
{\bf Input:}  Two acyclic loop automata $\mathcal{A}_1$, $\mathcal{A}_2$ over the input alphabet $A$.

\smallskip
\noindent
{\bf Question:} Does $[L(\mathcal{A}_1)]_I \cap [L(\mathcal{A}_2)]_I \neq \emptyset$ hold?

\smallskip
\noindent
Aalbersberg and Hoogeboom~\cite{AaHo89} proved that for the trace monoid
$\dM(\mathsf{P4})$ the intersection nonemptiness problem for arbitrary finite
automata is undecidable. We use their technique to show: 
\begin{lem} \label{lemma-intersection-P4}
For $\dM(\mathsf{P4})$, intersection nonemptiness for acyclic loop automata
is {\sf NP}-hard.
\end{lem}

\begin{proof}
We give a reduction from 3SAT. Let $\varphi = \bigwedge_{i=1}^m C_i$ where for every $i\in[1,m]$, 
$C_i = (L_{i,1} \vee L_{i,2} \vee L_{i,3})$ is a clause consisting of three literals. Let $x_1, \ldots, x_n$ be 
the boolean variables that occur in $\varphi$. In particular, every literal $L_{i,j}$ belongs to the set $\{x_1, \ldots, x_n, \neg x_1, \ldots, \neg x_n \}$.

Let $p_1, p_2, \ldots, p_n$ be a list of the first $n$ prime numbers. So, for each boolean variable $x_i$ we have the 
corresponding prime number $p_i$. We encode a valuation $\beta \colon \{x_1, \ldots, x_n \} \to \{0,1\}$ by any natural number
$N$ such that $N \equiv 0 \bmod p_i$ if and only if $\beta(x_i) = 1$.
For a positive literal $x_i$ let $S(x_i) = \{ p_i \cdot n \mid n \in \mathbb{N}\}$ and for a negative
literal $\neg x_i$ let $S(\neg x_i) = \{ p_i \cdot n + r \mid n \in \mathbb{N}, r\in[1,p_i-1]\}$.
Moreover, for every $i\in[1,m]$ let
$
S_i = S(L_{i,1}) \cup S(L_{i,2}) \cup S(L_{i,3}) . 
$
Thus, $S_i$ is the set of all numbers that encode a valuation, which makes the clause $C_i$ true.
Hence, the set
$
S = \bigcap_{i=1}^n S_i
$
encodes the set of all valuations that make $\varphi$ true.

We first construct an acyclic loop automaton $\mathcal{A}_1$ with
\[
L(\mathcal{A}_1) =  \prod_{i=1}^m  \{ a  (bc)^{N_i}  d  \mid N_i \in S_i \} .
\]    
Note that $\varphi$ is satisfiable iff $[L(\mathcal{A}_1)]_I$ contains a trace from   $[\{ (a  (bc)^N  d)^m \mid N \in \mathbb{N}\}]_I$.
We will ensure this property with a second acyclic loop automaton $\mathcal{A}_2$ that satisfies the equality
$
L(\mathcal{A}_2) = b^* (ad (bc)^*)^{m-1} ad c^* .
$

We claim that $[L(\mathcal{A}_1)]_I \cap [L(\mathcal{A}_2)]_I = [\{ (a  (bc)^N  d)^m \mid N \in S\}]_I$.
First assume that $w \equiv_I (a  (bc)^N  d)^m$ for some $N \in S$. 
We have 
$$
w \equiv_I (a  (bc)^N  d)^m \equiv_I b^N (ad (bc)^N)^{m-1} ad c^N
$$
and thus $[w]_I \in  [L(\mathcal{A}_2)]_I$. 
Moreover, since $N \in S$ we get $[w]_I \in  [L(\mathcal{A}_1)]_I$. 
For the other direction, let $[w]_I \in [L(\mathcal{A}_1)]_I \cap [L(\mathcal{A}_2)]_I$.
Thus 
$$
w \equiv_I  \prod_{i=1}^m (a  (bc)^{N_i}  d) \equiv b^{N_1} \bigg( \prod_{i=1}^{m-1} (ad  c^{N_i} b^{N_{i+1}}) \bigg) ad c^{N_m} ,
$$
where $N_i \in S_i$ for $i\in[1,m]$. Moreover, the fact that $[w]_I \in [L(\mathcal{A}_2)]_I$
means hat there are $k_0, k_1, \ldots, k_{m-1}, k_m \geq 0$ with
\begin{align*}
b^{N_1} \bigg( \prod_{i=1}^{m-1} (ad  c^{N_i} b^{N_{i+1}}) \bigg) ad c^{N_m}  & \equiv_I 
 b^{k_0} \bigg( \prod_{i=1}^{m-1} (ad (bc)^{k_i}) \bigg) ad c^{k_m} \\
&\equiv_I  b^{k_0} \bigg( \prod_{i=1}^{m-1} (ad b^{k_i} c^{k_i}) \bigg) ad c^{k_m} .
\end{align*}
Since every symbol is dependent from $a$ or $d$, this identity implies $N_i = N_{i+1}$ for $i\in[1,m-1]$.
Thus, $[w]_I \in  [\{ (a  (bc)^N  d)^m \mid N \in S\}]_I$.
\end{proof}
For a graph group $\dG(A,I)$ the  \emph{membership problem for acyclic loop automata}
is the following computational problem:
 
\smallskip
\noindent
{\bf Input:}  An acyclic loop automaton $\mathcal{A}$  over the input alphabet $A \cup A^{-1}$.

\smallskip
\noindent
{\bf Question:} Is there a word $w \in L(\mathcal{A})$ such that $w=1$ in $\dG(A,I)$?

\smallskip
\noindent

It is straightforward to reduce the intersection nonemptiness problem for acyclic loop automata over $\dM(A,I)$
to the membership problem for acyclic loop automata over $\dG(A,I)$.
For the rest of this section let $\Sigma =  \{a,b,c,d,a^{-1},b^{-1},c^{-1},d^{-1}\}$ and 
let $\theta \colon \Sigma^* \to \dG(P_4)$ be the canonical homomorphism that maps a word over $\Sigma$
to the corresponding group element.
\begin{lem} \label{lemma-acyclic-loop-P4}
For $\dG(\mathsf{P4})$, the membership problem for acyclic loop automata
is {\sf NP}-hard.
\end{lem}
\begin{proof}
The lemma follows easily from Lemma~\ref{lemma-intersection-P4}.
Note that $[L(\mathcal{A}_1)]_I \cap [L(\mathcal{A}_2)]_I \neq \emptyset$ if and only if $1 \in \theta(L(\mathcal{A}_1) L(\mathcal{A}_2)^{-1})$ in the 
graph group $\dG(\mathsf{P4})$.
Moreover, it is straightforward to construct from acyclic loop automata $\mathcal{A}_1$ and $\mathcal{A}_2$ an acyclic loop automaton for 
$L(\mathcal{A}_1) L(\mathcal{A}_2)^{-1}$. We only have to  replace every transition label $w$ in  $\mathcal{A}_2$ by $w^{-1}$, then reverse all transitions in $\mathcal{A}_2$ 
and concatenate the resulting automaton with $\mathcal{A}_1$ on the left.
\end{proof}

We can now use a construction from \cite{LohSte08} to reduce knapsack to membership for acyclic loop automata.
\begin{lem} \label{lemma-knapsack-P4}
Knapsack for the graph group $\dG(\mathsf{P4})$ is {\sf NP}-hard.
\end{lem}
\begin{proof}
By Lemma~\ref{lemma-acyclic-loop-P4} it suffices to reduce for $\dG(\mathsf{P4})$
the membership problem for  acyclic loop automata to knapsack.
Let $\mathcal{A} = (Q, \Sigma,\Delta, q_0, q_f)$ be an acyclic loop automaton with transitions
$\Delta \subseteq Q \times  \Sigma^* \times Q$. W.l.o.g. assume that $Q = \{1, \ldots, n\}$.

We reuse a construction from \cite{LohSte08},
where the rational subset membership problem for $\dG(\mathsf{P4})$ was reduced to the submonoid
membership problem for $\dG(\mathsf{P4})$. 
For a state $q \in Q$ let $\widetilde{q} =  (ada)^q d (ada)^{-q} \in \Sigma^*$.
Let us fix the morphism $\varphi \colon \Sigma^* \to \Sigma^*$ with $\varphi(x) = xx$ for $x \in \Sigma$.
For a transition
$t = (p,w,q) \in \Delta$ let $\tilde{t} = \widetilde{p} \,\varphi(w)\, \widetilde{q}^{-1}$ and define
$S=\{ \tilde{t} \mid t \in \Delta \}^*$. 
In \cite{LohSte08} it was shown that $1 \in \theta(L(A))$ if and only if $\theta(\widetilde{q_0} \, \widetilde{q_f}^{-1}) \in \theta(S)$.

We construct in polynomial time a knapsack instance over $\dG(P_4)$ from the automaton $\mathcal{A}$ as follows:
Let us choose an enumeration $t_1, t_2, \ldots, t_m$ of the transitions of $\mathcal{A}$ such that the following holds, 
where $t_i = (p_i, w_i, q_i)$: If $q_j = p_k$ then $j \leq k$. Since $\mathcal{A}$
is an acyclic loop automaton such an enumeration exists.  
The following claim proves the theorem.

\medskip
\noindent
{\em Claim:} $1 \in \theta(L(\mathcal{A}))$ if and only if $\theta(\widetilde{q_0}\, \widetilde{q_f}^{-1}) \in \theta( \tilde{t}_1^*\, \tilde{t}_2^* \cdots \tilde{t}_m^*)$.

\medskip
\noindent
One direction is clear: If  $\theta(\widetilde{q_0}\, \widetilde{q_f}^{-1}) \in \theta( \tilde{t}_1^*\, \tilde{t}_2^* \cdots \tilde{t}_m^*)$, then 
 $\theta(\widetilde{q_0} \,\widetilde{q_f}^{-1}) \in \theta(S)$. Hence, by \cite{LohSte08} we have $1 \in \theta(L(\mathcal{A}))$. On the other hand, if 
 $1 \in \theta(L(\mathcal{A}))$, then there exists a path in $\cA$ of the form
 $$
 q_0 = s_0 \xrightarrow{a_1} s_1 \xrightarrow{a_2} s_2 \cdots s_{k-1} \xrightarrow{a_k} s_k  = q_f 
 $$   
 such that $\theta(a_1 a_2 \cdots a_k) = 1$. Let $(s_{j-1},a_j,s_j) = t_{i_j}$, where we refer to the above enumeration of all transitions. 
 Then, we must have $i_1 \leq i_2 \leq \cdots \leq i_k$. Moreover, we have 
 $$\theta(\widetilde{q_0}\, \widetilde{q_f}^{-1}) =  \theta(\widetilde{q_0} \, a_1 a_2 \cdots a_k \, \widetilde{q_f}^{-1}) = \theta(\tilde{t}_{i_1} \tilde{t}_{i_2} \cdots \tilde{t}_{i_k}) 
 \in  \theta( \tilde{t}_1^*\, \tilde{t}_2^* \cdots \tilde{t}_m^*).$$
 This proves the claim and hence the theorem.
\end{proof}

\begin{thm}
Let $(A,I)$ be an independence alphabet, which is not a transitive forest. Then,
knapsack for the graph group $\dG(A,I)$ is {\sf NP}-complete.
\end{thm}

\begin{proof}
If $(A,I)$ is not a transitive forest, then $\mathsf{P4}$ or $\mathsf{C4}$ is an induced subgraph of $(A,I)$ \cite{Wolk65}. 
Thus, $\dG(\mathsf{P4})$ or $\dG(\mathsf{C4}) \cong F_2 \times F_2$ is a subgroup of $\dG(A,I)$. Hence,
{\sf NP}-hardness of knapsack for $\dG(A,I)$ follows from \cite{LohreyZ16} or Lemma~\ref{lemma-knapsack-P4}.
\end{proof}

\section{Open problems}

Our proof for the  {\sf NP}-completeness of knapsack for $\dG(A,I)$ if $(A,I)$ is not a transitive forest makes use of solutions with 
exponentially large values for exponent variables. As a consequence, we cannot adapt our proof to the subset sum problem.
For $\dG(\mathsf{C4}) \cong F_2 \times F_2$ we have shown in \cite{LohreyZ16} that the subset sum problem is   {\sf NP}-complete.
But for $\dG(\mathsf{P4})$ it remains open whether subset sum is  {\sf NP}-complete.

Recall from Remark~\ref{remark} that the knapsack variant, where we ask for a solution in $\Z^k$ (let us call this variant 
\emph{integer-valued knapsack})  can be reduced to our version of knapsack,
where the solution is restricted to $\N^k$. Hence, there  might be groups where integer-valued knapsack is easier than
standard knapsack. Indeed, in the introduction we remarked
that integer-valued knapsack for the group $\Z$ with binary encoded integers 
can be solved in polynomial time. This motivates the investigation of integer-valued knapsack for other 
graph groups. In \cite{LohreyZ16} we proved that integer-valued knapsack for $\dG(\mathsf{C4})$ is {\sf NP}-complete,
even if group elements are given by uncompressed words. On the other hand, our {\sf NP}-hardness proof for $\dG(\mathsf{P4})$
makes essential use of the fact that solution vectors are restricted to $\N^k$, and it remains open whether integer-valued knapsack for
$\dG(\mathsf{P4})$ is  {\sf NP}-hard. In this context, it is interesting to note that there are some interesting differences between
$\dG(\mathsf{P4})$ and $\dG(\mathsf{C4})$ with respect to algorithmic problems. For instance, whereas the subgroup membership
problem for $\dG(\mathsf{C4})$ is undecidable \cite{Mih66}, the subgroup membership problem for $\dG(\mathsf{P4})$ (and any graph group
$\dG(\Gamma)$ with $\Gamma$ chordal) is decidable \cite{KaWeMy05,Loh09ProcAMS}.

As remarked in the introduction, there are nilpotent groups with an undecidable
knapsack problem. If a nilpotent group is an automatic group, then it must be virtually abelian \cite{EpsCHLPT92}.
This raises the question whether knapsack is decidable for every automatic group. 

\bibliography{bib}
\bibliographystyle{plain}

\end{document}